\documentclass[a4paper]{article}
\usepackage{amsthm,amsmath,amssymb}
\usepackage{epsfig}

\newtheorem{assumption}{Assumption}
\newtheorem{remark}{Remark}
\newtheorem{theorem}{Theorem}
\newtheorem{lemma}{Lemma}

\title{Regularity analysis for stochastic partial
differential equations with nonlinear
multiplicative trace class noise}

\author{Arnulf Jentzen$^1$ 
and 
Michael R\"{o}ckner$^2$
\bigskip
\\
\small{$^1$Program in Applied 
and Computational Mathematics, 
Princeton University,}
\\
\small{Fine Hall,
Washington Road, 
Princeton, New Jersey 08544, 
USA, 
e-mail:
jentzen@math.princeton.edu}
\smallskip
\\
\small{$^2$Faculty 
of Mathematics, 
Bielefeld University,
33501 Bielefeld, Germany,} 
\\
\small{e-mail: 
roeckner@math.uni-bielefeld.de}
}

\author{Arnulf Jentzen\thanks{Program 
in Applied
and Computational Mathematics,
Princeton University, 
Fine Hall, Washington Road,
Princeton,
NJ 08544-1000,
USA 
(ajentzen@math.princeton.edu)}
\  and
Michael R{\"o}ckner\thanks{Faculty of Mathematics,
Bielefeld University, 
Universitaetsstrasse~25,
33615~Bielefeld,
Germany 
(roeckner@math.uni-bielefeld.de) 
and 
Department of Mathematics 
and Statistics, Purdue 
University, 
150 N.~University St,
West Lafayette, 
IN 47907-2067,
USA (roeckner@math.purdue.edu)}}
\begin{document}

\maketitle

\begin{abstract}
In this article spatial and temporal regularity
of the solution process of a stochastic partial differential 
equation (SPDE) of evolutionary type with nonlinear
multiplicative trace class noise is analyzed.\\ \ \\
{\bf Key words:} stochastic partial differential equations, regularity analysis, nonlinear multiplicative noise\\ \ \\
{\bf AMS subject classification:} 60H15,
35R60
\end{abstract}
\section{Introduction}\label{intro}
Spatial and temporal regularity of the solution process
of a stochastic partial differential equation (SPDE) of
evolutionary type are investigated in this article. More
precisely, it is analyzed under which conditions on the
noise term of a semilinear SPDE the solution process
enjoys values in the domains of fractional powers of the
dominating linear operator of the SPDE. It turns out 
that the essential constituents determining the regularity
of the solution process are assumptions on the covariance
operator of the driving noise process of the SPDE and
appropriate boundary conditions on the diffusion coefficient.
While the regularity of 
(affine) linear SPDEs 
has been intensively 
studied in previous results 
(see, e.g., N.~V.~Krylov \& 
B.\ L.\ Rozovskii \cite{kr79},
B.\ L.\ Rozovskii \cite{r90},
G.\ Da Prato \& J. Zabczyk \cite{dz92},
N.\ V.\ Krylov \cite{k94},
Z.\ Brze{\'z}niak~\cite{b97b},
Z.\ Brze{\'z}niak \& J. van 
Neerven~\cite{bv03},
S.\ Tindel et al.~\cite{ttv04}
and 
Z.\ Brze{\'z}niak et al.~\cite{bvvw08}), 
the main purpose of 
this article is to handle possibly nonlinear diffusion coefficients
in SPDEs driven by trace class 
Brownian noise 
(see also X.\ Zhang~\cite{z07}
for a related result).

In order to illustrate the results in this article, we 
concentrate on the following example SPDE in this introductory
section and refer to 
Section~\ref{sec:setting} for our 
general setting and to 
Section~\ref{sec:examples} for 
further examples of SPDEs. 
Let $ T \in (0,\infty) $, 
let $ ( \Omega, 
\mathcal{F}, \mathbb{P} ) $ be a probability space with
a normal filtration $ \left( \mathcal{F}_t \right)_{ 
t \in [0,T] } $ and let $ H = L^2( (0,1), \mathbb{R}) $
be the $ \mathbb{R} $-Hilbert space of equivalence classes
of square integrable functions from $ (0,1) $ to $ \mathbb{R} $.
Moreover, let 
$ f, b \colon (0,1) \times 
\mathbb{R} \rightarrow 
\mathbb{R} 
$ 
be two continuously 
differentiable functions
with globally bounded derivatives, 
let 
$ 
  x_0 \colon (0,1) 
  \rightarrow \mathbb{R} 
$ 
be a smooth function with $ \lim_{x \searrow 0} x_0(x) = \lim_{x \nearrow 1} 
x_0(x) = 0 $ and let 
$ 
W \colon [0,T] \times 
\Omega \rightarrow
H 
$ be a standard 
$ Q $-Wiener process 
with respect to 
$ \left( \mathcal{F}_t \right)_{ t \in [0,T] } $ with a
covariance operator $ Q \colon H \rightarrow H $. It is a classical
result (see, e.g., Theorem VI.3.2 in \cite{w05a}) that the 
covariance operator $ Q \colon H \rightarrow H $ of the Wiener
process $ W \colon [0,T] \times \Omega \rightarrow H $ has an
orthonormal basis $ g_j \in H $, $ j \in \{0,1,2,\ldots\} $,
of eigenfunctions with summable eigenvalues $ \mu_j \in [0,
 \infty) $, $ j \in \{0,1,2,\ldots\} $. In order to have a
more concrete example, we consider the choice $ g_0(x) = 1 $,
$ g_j( x ) = \sqrt{2} \cos( j \pi x ) $, $ \mu_0 = 0 $ and
$ \mu_j = j^{-r} $ for all $ x \in (0,1) $ and all $ j \in
\mathbb{N} $ with a given real number $ r \in (1,\infty) $
in the following and refer to 
Section~\ref{sec:examples} for
possible further examples. 
Then we consider the SPDE
\begin{equation}\label{eq:SPDE1}
  d X_t(x)
  =
  \left[
    \frac{ \partial^2 }{ \partial x^2 } \, X_t(x)
    + f( x, X_t(x) )
  \right] dt
  +
  b( x, X_t(x) ) \, dW_t(x)
\end{equation}
with $ X_t(0) = X_t(1) = 0 $ and $ X_0(x) = 
x_0(x) $ for $ t \in [0,T] $ and $ x \in (0,1) $.
Under the assumptions above the SPDE \eqref{eq:SPDE1} has a unique mild solution. Specifically, there exists an up to indistinguishability unique adapted
stochastic process 
$ 
  X
  \colon [0,T] \times 
  \Omega \rightarrow H 
$ 
with continuous sample paths which 
satisfies
\begin{equation}\label{eq:Dprocess_X}
  X_t
  =
  e^{ A t } x_0
  +
  \int_0^t
  e^{ A( t - s ) } F( X_s ) \, ds
  +
  \int_0^t
  e^{ A( t - s ) } B( X_s ) \, dW_s
\end{equation}
$ \mathbb{P} $-a.s.\ for all 
$ t \in [0,T] $ where 
$ A \colon D( A ) \subset H 
\rightarrow H $ is the 
Laplacian with Dirichlet 
boundary conditions and 
where 
$ F \colon H \rightarrow H $ 
and 
$ B \colon H \rightarrow HS(U_0,H) $ 
are given by $ ( F( v ) )( x ) = f( x, v( x ) ) $ and $ ( B( v ) u )( x ) = b( x, v( x ) ) \cdot u( x ) $ for all $ x \in (0,1) $, $ v \in H $ 
and all $ u \in U_0 $. 
Here $ U_0 = Q^{ 1/2 }( H ) 
$ with
$ \left< v, w \right>_{ U_0 } 
= \left< Q^{ - 1 / 2 } v, 
Q^{ - 1 / 2 } w \right>_H $ 
for all
$ v, w \in U_0 $
is the image $ \mathbb{R} $-Hilbert 
space of $ Q^{ \frac{1}{2} } $ 
(see Appendix C in \cite{pr07}).

We are then interested to know for which 
$ \gamma \in [0,\infty) $
in dependence on the decay rate 
$ r \in (1,\infty) $ of the
eigenfunctions of the covariance operator 
$ 
  Q \colon H \rightarrow H 
$
the solution process 
$ 
  X \colon [0,T] \times \Omega 
  \rightarrow H 
$
of \eqref{eq:SPDE1} takes values in $ D( (-A)^{ \gamma } ) $.
For the SPDE \eqref{eq:SPDE1} it turns out that
\begin{equation}\label{eq:result}
  \mathbb{P}\Big[
    X_t \in D( (-A)^{ \gamma } )
  \Big]
  = 1
\end{equation}
holds for all $ t \in [0,T] $ and all 
$ \gamma \in [0, \frac{\min(3, r+1)}{4}) $
(see Theorem~\ref{thm:existence} 
in~Section~\ref{sec:main}
for the main result of
this article 
and 
Subsection~\ref{sec:example1} 
for the SPDE \eqref{eq:SPDE1}). 
Under further assumptions 
on the diffusion
coefficient function 
$ 
  b \colon (0,1) \times \mathbb{R} 
  \rightarrow \mathbb{R} 
$, 
the solution 
of \eqref{eq:SPDE1} has 
even more regularity which 
can be seen in 
Subsection~\ref{sec:example2}.

In the following we relate the results in this article
with existing regularity results in the literature and also
illustrate how \eqref{eq:result} 
can be established.
The regularity of linear SPDEs has been intensively 
analyzed in the literature (see, e.g., \cite{kr79,r90,dz92,k94,bv03,b97b,ttv04,bvvw08}). For
instance, in Theorem 6.19 in \cite{dz92}, Da Prato and Zabczyk
already showed for the SPDE \eqref{eq:SPDE1} in the case 
$ f( x,y ) = 0 $ for all $ x \in (0,1) $, $ y \in \mathbb{R} $
and $ b : (0,1) \times \mathbb{R} \rightarrow \mathbb{R} $
sufficiently small and linear in the second variable that
\eqref{eq:result} holds for all $ t \in [0,T] $ and all
$ \gamma \in [0, \frac{ \min( 4, r+1 ) }{ 4 } ) $.
Their key idea in Theorem~6.19 
in \cite{dz92} was to apply the
Banach fixed point theorem 
in an appropriate Banach space of
$ D( (-A)^{ \gamma } ) $-valued 
stochastic processes for
$ \gamma \in [0, \frac{ \min( 4, r+1 ) }{ 4 } ) $.
This approach is based on the fact that
$ B \colon 
H \rightarrow HS( U_0, H ) $ is linear and 
globally Lipschitz continuous from $ D( (-A)^{ \gamma } ) 
\subset H $ to $ HS( U_0, D( (-A)^{ \gamma } ) ) \subset
HS( U_0, H ) $ for $ \gamma \in [0, \frac{ \min( 2, r-1 ) }{ 4 } ) $
since $ b : (0,1) \times \mathbb{R} \rightarrow \mathbb{R} $ is
assumed to be linear in its second variable. Although their
method in Theorem 6.19 in \cite{dz92} works quite well for
linear SPDEs, it can not be generalized to nonlinear SPDEs of the form
\eqref{eq:SPDE1}.
More formally, in the case of 
a nonlinear 
$ b \colon (0,1) \times 
\mathbb{R} \rightarrow \mathbb{R} 
$,
$ 
B \colon 
H \rightarrow HS( U_0, H ) 
$ 
is in general
not globally Lipschitz continuous from $ D( (-A)^{\gamma} ) $
to $ HS( U_0, D( (-A)^{\gamma} ) ) $ for $ \gamma > 0 $
although 
$ b \colon (0,1) \times 
\mathbb{R} \rightarrow \mathbb{R} $
is assumed to have globally bounded derivatives. Therefore, a
contraction argument as in 
Theorem~6.19 in \cite{dz92} 
(see also J.\ van Neerven 
et al.~\cite{vvw08}
for a related result)
in a Banach space 
of $ D( (-A)^{\gamma} ) 
$-valued stochastic
processes for $ \gamma > \frac{1}{2} 
$ 
can in general not be established
for nonlinear SPDEs of the form \eqref{eq:SPDE1}. 
This difficulty is
a key problem of regularity 
analysis for nonlinear SPDEs 
and has been 
pointed out in X. Zhang \cite{z07} 
(see page 456 in \cite{z07}).

We now demonstrate our approach to analyze the regularity of
\eqref{eq:SPDE1} which overcomes the lack of Lipschitz
continuity of 
$ B \colon H \rightarrow HS( U_0, H ) $ with respect to $ D( (-A)^{ \gamma } ) $ and $ HS( U_0, D( (-A)^{\gamma} ) ) $ for $ \gamma > 0 $ in the nonlinear
case. First of all, by exploiting the smoothing effect of the
semigroup of the Laplacian in \eqref{eq:Dprocess_X}, the
existence of an up to modifications unique predictable
$ 
  D( (-A)^{ \gamma } ) 
$-valued 
solution process 
$ 
  X \colon [0,T] \times \Omega 
  \rightarrow D( (-A)^{ \gamma } ) 
$
of \eqref{eq:SPDE1} 
with
\begin{equation}\label{eq:boot1}
  \sup_{ t \in [0,T] }
  \mathbb{E}\!\left[
    \left\|
    X_t
    \right\|_{ D( (-A)^{ \gamma } ) 
    }^2
  \right]
  < \infty
\end{equation}
can 
be established 
immediately
for all 
$ \gamma \in [0, \frac{1}{2}) $
(see J.\ van Neerven 
et al.~\cite{vvw08} for details). 
However, we want to show \eqref{eq:result}
for all $ t \in [0,T] $ and all $ \gamma \in [0,\frac{\min(3,r+1)}{4} ) $ instead of $ \gamma \in [0,\frac{1}{2}) $. 
To this end a key estimate 
in our approach is 
the linear growth bound
\begin{equation}\label{eq:X2}
  \left\|
    B( v )
  \right\|_{ HS( U_0, 
    D( (-A)^{\alpha} ) ) }
  \leq 
  c_{ \alpha } 
  \left(
    1 + \left\| v \right\|_{ D( (-A)^{\alpha} ) }
  \right) 
\end{equation}
for all 
$ v \in D( (-A)^{\alpha} ) $,
$ \alpha \in 
[0, \frac{ \min( 1, r-1 ) }{ 4 } ) $ 
with
$ c_{ \alpha } \in [0,\infty) $,
$
  \alpha \in 
  [0, \frac{ \min(1,r-1) }{ 4 })
$, 
appropriate
which 
we sketch below.
We would like to point
out here that $ B \colon H
\rightarrow HS(U_0,H) $
fulfills the linear
growth bound~\eqref{eq:X2}
although it 
fails to be globally
Lipschitz continuous
from $ D( (-A)^\alpha ) $
to $ HS( U_0, D( (-A)^\alpha ) ) $
for $ \alpha > 0 $
in general
(see Section~\ref{sec:examples}
for the verification of
\eqref{eq:X2} 
in the case of SPDEs of the 
form~\eqref{eq:SPDE1}).
Exploiting 
estimate~\eqref{eq:X2} 
in an appropriate 
bootstrap argument 
will then show \eqref{eq:result}
for all $ t \in [0,T] $ and all $ \gamma \in [ 0, \frac{ \min( 3, r+1 ) }{ 4 } ) $. 
More formally,
using that 
the semigroup 
% $ e^{ A t } \in L( H ) $,
% $ t \in [0,T] $, 
is analytic 
with
$ e^{ A t }( H ) \subset
D(A) $
for all $ t \in (0,T] $
yields
\begin{align*}
  & 
  \int_0^t
  \mathbb{E}\Big[ 
    \big\|
    (-A)^{\gamma} \,
    e^{ A( t - s ) }
    B( X_s )
    \big\|_{ HS( U_0, H ) }^2
  \Big] \, ds
\\ & \leq 
  \int_0^t
  \big\|
    (-A)^{\vartheta} \,
    e^{ A( t - s ) }
  \big\|^2_{ L( H ) }
  \,
  \mathbb{E}\Big[
    \big\|
    (-A)^{(\gamma - \vartheta)}
    B( X_s )
    \big\|_{ HS( U_0, H ) }^2 
  \Big] \, ds
\\&\leq 
  \int_0^t
  \left( t - s 
  \right)^{ -2 \vartheta }
  \,
  \mathbb{E}\Big[ 
    \big\|
      B( X_s )
    \big\|_{ 
      HS( U_0, D( 
        (-A)^{ (\gamma - \vartheta) } 
      ) ) 
    }^2 
  \Big] \,
  ds
\end{align*}
and using 
estimate~\eqref{eq:X2}
then shows
\begin{align}\label{eq:newX1}
\nonumber
  &
  \int_0^t
  \mathbb{E}\Big[ 
    \big\|
    (-A)^{\gamma} \,
    e^{ A( t - s ) }
    B( X_s )
    \big\|_{ HS( U_0, H ) }^2
  \Big] \, ds
\\&\leq
\nonumber
  \int_0^t
  \left( t - s 
  \right)^{ - 2 \vartheta }
  \left| c_{ \gamma - \vartheta } 
  \right|^2 
  \mathbb{E}\bigg[
    \left( 
      1 + \left\| X_s \right\|_{ 
        D( (-A)^{ 
          (\gamma - \vartheta) } ) }
    \right)^{ \! 2 } \,
  \bigg] ds
\\&\leq
  2 
  \left| c_{ \gamma - \vartheta } 
  \right|^2
  \left(
    \int_0^t
    s^{ -2 \vartheta } \, ds
  \right)
  \left( 
    1 + 
    \sup_{ s \in [0,T] }
    \mathbb{E}\Big[ \left\| 
      X_s 
    \right\|_{ 
      D( 
        (-A)^{ (\gamma - \vartheta) 
      } ) }^2
    \Big]
  \right)
\\&\leq
\nonumber
  \frac{ 
    2 
    \left| c_{ \gamma - \vartheta } 
    \right|^2 
    \left( T + 1 \right) 
  }{ 
    ( 1 - 2 \vartheta ) 
  }
  \left( 
    1 + 
    \sup_{ s \in [0,T] }
    \mathbb{E}\Big[ \left\| 
      X_s 
    \right\|_{ 
      D( (-A)^{ (\gamma - \vartheta) } ) 
    }^2 \Big]
  \right)
  <
  \infty
\end{align}
for all $ t \in [0,T] $, 
$ \vartheta \in (\gamma - 
\frac{ \min(1,r-1) }{ 4 }, \frac{1}{2} ) $
and all $ \gamma \in [\frac{1}{2}, \frac{ \min(3, r+1 )}{4} ) $. We would like to point out that 
due to \eqref{eq:boot1}
the right hand side 
of \eqref{eq:newX1} is indeed 
finite.
Of course, \eqref{eq:newX1} 
then shows that $ \int_0^t e^{ A( t - s ) } B( X_s ) \, dW_s $,
$ t \in [0,T] $, 
has a modification with 
values in $ D( (-A)^{ \gamma } ) $
for all
$ \gamma 
\in 
[0, \frac{ \min(3, r+1 )}{4} ) $ 
and thus, \eqref{eq:result} holds 
for all $ t \in [0,T] $ and 
all 
$ 
  \gamma \in 
  [0, \frac{ \min(3, r+1 )}{4} ) 
$.
% To sum it up, the central idea here
% for proving \eqref{eq:result}
% for the SPDE~\eqref{eq:SPDE1}
% is the observation that
% $
%   B \colon H \rightarrow HS( U_0, H )
% $
% in \eqref{eq:Dprocess_X} 
% fulfills
% the linear growth 
% estimate~\eqref{eq:X2}
% although it in general
% fails to satisfy the
% corresponding Lipschitz
% estimate.

Regularities of nonlinear SPDEs as analyzed here have also been investigated 
in Zhang's instructive paper \cite{z07}. 
In contrast to the results in this
article, he investigated which conditions on the coefficients
and the noise of an SPDE suffice to ensure that the solution
process of the SPDE is infinitely often differentiable in 
the spatial variable, see Theorem 6.2 in \cite{z07}. The solution
process of \eqref{eq:SPDE1} 
in which we are interested 
is in general not 
twice differentiable
in the spatial variable
and therefore, Theorem 6.2 
in \cite{z07} can in general
not be applied to the 
SPDE~\eqref{eq:SPDE1} here. 

The rest of this article is organized as follows. In Section~\ref{sec:setting} the setting and assumptions used are formulated. Our main result, 
Theorem~\ref{thm:existence}, which states existence, uniqueness and regularity
of solutions of an SPDE 
with nonlinear multiplicative
trace class noise is presented 
in Section~\ref{sec:main}. This result is illustrated by various examples in Section~\ref{sec:examples}. 
The proof of Theorem~\ref{thm:existence} is postponed to the final section.
%To sum up, the regularity
%of nonlinear SPDEs of the form \eqref{eq:SPDE1} has not yet %been analyzed to the best of my knowledge.
%
%
\section{Setting and assumptions}
\label{sec:setting}
Throughout this article assume 
that the following setting
is fulfilled.

Let $ T \in (0, \infty) $ be a
real number, let 
$ \left( \Omega, \mathcal{F}, \mathbb{P} \right) $ 
be a probability space 
with a normal filtration
$ ( \mathcal{F}_t )_{ t \in [0,T] } $
and let
$ 
  \left( 
    H, 
    \left< \cdot , \cdot \right>_H, 
    \left\| \cdot \right\|_H
  \right) $ 
and
$ 
  \left( 
    U, 
    \left< \cdot , \cdot \right>_U, 
    \left\| \cdot \right\|_U 
  \right) 
$ 
be two separable
$ \mathbb{R} 
$-Hilbert spaces.
Moreover, let 
$ Q \colon U \rightarrow U $ 
be a trace class operator and 
let 
$ 
  W \colon [0,T] 
  \times \Omega \rightarrow U 
$ 
be a standard $ Q $-Wiener 
process with respect 
to $ ( \mathcal{F}_t )_{ t \in [0,T] } $.
\begin{assumption}[Linear 
operator A]\label{semigroup}
Let $ A \colon D(A) \subset H
\rightarrow H $ be a closed and densely
defined linear operator
which generates a stronlgy 
continuous analytic semigroup
$ e^{ A t } \in L(H) $, $ t \in [0,\infty) $.
\end{assumption}
Let $ \eta \in [0,\infty) $ be a nonnegative real number such that
$
  \sigma(A)
  \subset
  \{ 
    \lambda \in \mathbb{C}
    \colon
    \text{Re}( \lambda ) < \eta
  \}
$ where 
$ \sigma(A) \subset \mathbb{C} 
$ denotes 
as usual the spectrum
of the linear
operator 
$ A \colon D(A) \subset
H \rightarrow H $. 
Such a real number exists
since $ A $ is assumed to be a generator
of a strongly continuous semigroup
(see Assumption~\ref{semigroup}).
By 
$ V_r := 
D\!\left( 
  \left( \eta - A \right)^r 
\right) 
\subset H 
$
equipped with the norm
$ 
  \left\| v \right\|_{ V_r }
  := 
  \left\| 
    \left( \eta - A \right)^r \! v 
  \right\|_H 
$
for all $ v \in V_r $ 
and all $ r \in [0,\infty) $ we denote
the $ \mathbb{R} $-Hilbert spaces
of domains of fractional powers 
of the linear operator 
$ \eta - A \colon D(A) \subset H
\rightarrow H $
(see, e.g., Subsection~11.4.2
in Renardy and
Roggers~\cite{rr93}).
\begin{assumption}[Drift term $F$]\label{drift}
Let 
$ F \colon H \rightarrow H $ 
be a globally
Lipschitz continuous mapping.
\end{assumption}
In order to formulate the assumption on the 
diffusion coefficient of our SPDE, we denote by
$ 
  \left( 
    U_0, 
    \left< \cdot , \cdot \right>_{ U_0 }, 
    \left\| \cdot \right\|_{ U_0 }
  \right) 
$ 
the separable
$\mathbb{R}$-Hilbert space 
$ U_0 := Q^{ 1/2 }( U ) $
with $ \left< v, w \right>_{ U_0 } = \left< Q^{ - 1 / 2 } v, Q^{ - 1 / 2 } 
w \right>_U $ for all $ v, w \in U_0 $
(see, for example, 
Subsection~2.3.2 
in \cite{pr07}).
Here 
$ Q^{-1/2} \colon \text{im}(Q^{1/2})
\subset U \rightarrow U $ 
denotes the pseudo inverse of 
%the bounded linear operator
$ Q^{1/2} \colon U \rightarrow U $
(see, e.g., Appendix~C in \cite{pr07}
for details).
\begin{assumption}[Diffusion term $B$]\label{diffusion}
Let 
$ B \colon H \rightarrow HS(U_0,H) $ 
be a globally Lipschitz continuous
mapping and let $ \alpha \in [0,\frac{1}{2}) $,
$ c \in [0,\infty) $ be
real numbers such
that $ B( V_{ \alpha } ) \subset HS( U_0, V_{ \alpha } ) $
and
% \begin{equation}\label{condition_AA}
$
  \left\| B(v) \right\|_{ HS( U_0, V_{ \alpha } ) } 
  \leq
  c \left( 1 + \left\| v \right\|_{ V_{ \alpha } }
  \right) 
$
% \end{equation}
for all $ v \in V_{ \alpha } $.
\end{assumption}
\begin{assumption}[Initial value $\xi$]\label{initial}
Let $ \gamma \in [\alpha, \frac{1}{2} + \alpha ) $,
$ p \in [2,\infty)$ and 
let
$ 
  \xi \colon \Omega 
  \rightarrow V_{\gamma} 
$ 
be an 
$ \mathcal{F}_0 
$/$ 
\mathcal{B}\left( V_{\gamma} \right) 
$-measurable 
mapping with 
$ 
  \mathbb{E}\big[ 
    \| \xi \|^p_{   
      V_{\gamma} 
    } 
  \big] 
  < \infty 
$.
\end{assumption}
Some examples satisfying 
Assumptions \ref{semigroup}-\ref{initial} 
are presented 
in 
Section~\ref{sec:examples}.
%
%
%
% Assumptions \ref{semigroup}-\ref{initial}
% are more or less straightforward to
% check except of the linear growth
% bound in $ V_{\alpha} $, i.e.
% $$
%   \left\| B(v) \right\|_{ HS( U_0, V_{ \alpha } ) } 
%   \leq
%   c \left( 1 + \left\| v \right\|_{ V_{ \alpha } }
%   \right) 
% $$
%
%
%
\section{Main result}\label{sec:main}
The assumptions in 
Section~\ref{sec:setting} suffice 
to ensure the existence of a 
unique $ V_{ \gamma } $-valued
solution of the SPDE~\eqref{eq:SPDE}.
\begin{theorem}[Existence and regularity of 
the solution]\label{thm:existence}
Assume that the setting
in Section~\ref{sec:setting} is
fulfilled. Then there exists 
an up to modifications 
unique predictable stochastic process
$
  X \colon [0,T] 
  \times \Omega \rightarrow V_{\gamma} 
$
which fulfills
$
  \sup_{ t \in [0,T] } 
  \mathbb{E}\big[
    \|
      X_t
    \|_{ V_{ \gamma } }^p 
  \big]
  < \infty
$,
$
  \sup_{ t \in [0,T] } 
  \mathbb{E}\big[
    \|
      B(X_t)
    \|_{ HS(U_0,V_{ \alpha }) }^p
  \big]
  < \infty
$ and
\begin{equation}
\label{eq:SPDE}
%  \mathbb{P}\left[
    X_t
    = 
    e^{At} \xi 
    +
    \int_0^t e^{A(t-s)} F(X_s) \, ds 
    + 
    \int_0^t e^{A(t-s)} B(X_s) \, dW_s
%   \right]
%   =
%   1 
\end{equation}
$ \mathbb{P} $-a.s.\ for 
all $t \in [0,T]$.
Moreover, we have
\begin{equation}\label{eq:holder}
  \sup_{ 
    \substack{
      t_1, t_2 \in [0,T] 
    \\
      t_1 \neq t_2
    }
  }
  \frac{
    \left(
      \mathbb{E}\big[
      \left\| X_{ t_2 } - X_{ t_1 }
      \right\|_{ V_{ r } }^p
      \big]
    \right)^{ \frac{1}{p} }
  }{
    \left| t_2 - t_1 
    \right|^{ \min( \gamma - r, \frac{1}{2} ) }
  }
< \infty
\end{equation}
for every $ r \in [0,\gamma) $.
Additionally, the solution 
process
$ X_t$, $t \in [0,T], $
is even continuous with respect to
$ 
  \big( \mathbb{E}\big[ 
      \| \! \cdot \!
      \|_{ V_{ \gamma } }^p
    \big]
  \big)^{ \frac{1}{p} } 
$.
\end{theorem}
The proof of Theorem \ref{thm:existence} 
is given in 
Section~\ref{secproofs}. 
The parameters 
$ \alpha \in [0,\frac{1}{2}) $,
$ \gamma \in [\alpha,\frac{1}{2} + \alpha) $
and 
$ p \in [2,\infty) $ used in 
Theorem \ref{thm:existence}
are given in Assumptions \ref{diffusion}
and \ref{initial}.

Estimate~\eqref{eq:holder}
and the continuity of 
the solution process
$ X_t $, $ t \in [0,T] $, 
with respect to 
$ 
  \big( \mathbb{E}\big[ 
      \| \! \cdot \!
      \|_{ V_{ \gamma } }^p
    \big]
  \big)^{ \frac{1}{p} } 
$
as asserted
in Theorem~\ref{thm:existence}
can also be written as
\begin{equation}
  X \in 
  \cap_{ r \in [0,\gamma] } \,
  \mathcal{C}^{ \min(\gamma-r,\frac{1}{2}) 
  }\big( [0,T], 
  L^p( \Omega ; V_r )
  \big) .
\end{equation}
Let us complete this section with
the following remarks.

In this article we investigate
predictable 
$ V_{ \gamma } $-valued
solution processes
of the SPDE~\eqref{eq:SPDE}.
For results analyzing 
continuity of sample paths
for 
$ H $-valued
solution processes
of SPDEs of the form~\eqref{eq:SPDE},
the reader is referred to
P.~Kotelenez~\cite{k82}
and
L.~Tubaro~\cite{t84}, for instance.

If the initial value $ X_0 = \xi $ 
of the SPDE~\eqref{eq:SPDE} above
is $H$-valued
only, then $X_t$
takes values in $ V_r $ for 
all $ r < \frac{1}{2} + \alpha $
and all $ t \in (0,T] $
nevertheless.
More formally, if 
Assumptions~\ref{semigroup}-\ref{diffusion}
are fulfilled and
if $ \xi \colon \Omega \rightarrow H $
is an 
$ \mathcal{F}_0 
$/$ \mathcal{B}(H)
$-measurable 
mapping 
with 
$ 
  \mathbb{E}\big[
  \| \xi \|_H^p \big]
  < \infty 
$ 
for some $ p \in [2,\infty) $,
then Theorem~\ref{thm:existence}
shows the existence of a predictable
solution process 
$ X \colon [0,T] \times \Omega 
\rightarrow H $ of 
$ \eqref{eq:SPDE} $ and this
process additionally satisfies
$ \mathbb{P}\left[ 
X_t \in V_r \right]=1 $
with 
$ \mathbb{E}\big[ 
\| X_t 
\|_{ V_r }^p
\big] < \infty $ for all
$ r \in 
[0,\frac{1}{2}+\alpha) $
and all $ t \in (0,T] $.

\section{Examples}\label{sec:examples}
In this section Theorem \ref{thm:existence} is illustrated with various examples. 
To this end let $ d \in \mathbb{N} $ 
and let 
$ H = U = L^2( (0,1)^d, \mathbb{R} ) $ 
be the $ \mathbb{R} $-Hilbert space 
of equivalence classes 
of $ \mathcal{B}( (0,1)^d ) 
$/$ \mathcal{B}( \mathbb{R} ) $-measurable 
and Lebesgue square integrable functions from $ (0,1)^d $ to $ \mathbb{R} $. 
As usual we do not distinguish between a
square integrable function from $ (0,1)^d $
to $ \mathbb{R} $ and its
equivalence class in $ H $.
For simplicity we restrict our attention to the domain $ (0,1)^d $ although more complicated domains in $ \mathbb{R}^d $ could be considered. The scalar product and the norm in $ H $ and $ U $ are given by
\[
  \left< v,w \right>_H
  =
  \left< v,w \right>_U
  =
  \int_{ (0,1)^d }
  v( x ) \cdot w( x ) \, dx
\]
and
\[
  \left\| v \right\|_H
  =
  \left\| v \right\|_U
  =
  \left(
    \int_{ (0,1)^d }
    \left| v( x ) \right|^2 dx
  \right)^{ \! \frac{1}{2} }
\]
for all $ v,w \in H = U $. Moreover, the Euclidean norm $ \left\| x \right\|_{ \mathbb{R}^d } := ( | x_1 |^2 + \ldots + | x_d |^2 )^{ \frac{1}{2} } $ for all $ x = ( x_1, \ldots, x_d ) \in \mathbb{R}^d $ is used here. 
Additionally, the notations
\[
  \left\| v \right\|_{ C\left( (0,1)^d, \mathbb{R} \right) }
  :=
  \sup_{ x \in (0,1)^d }
  \left| v( x ) \right| \in [0,\infty]
\]
and
\[
  \left\| v \right\|_{ C^r\left( (0,1)^d, \mathbb{R} \right) }
  :=
  \sup_{ x \in (0,1)^d }
  \left| v( x ) \right|
  +
  \sup_{ 
    \substack{
      x, y \in (0,1)^d 
    \\ 
      x \neq y 
    } 
  }
  \frac{ \left| v( x ) - v( y ) \right| }{ \left\| x - y \right\|^r_{ \mathbb{R}^d } }  \in [0,\infty]
\]
for all $ r \in (0,1] $ and 
all functions 
$ v \colon (0,1)^d \rightarrow \mathbb{R} $ 
are used in this section. 
We also define 
\begin{multline*}
  \left\| v \right\|_{ W^{ r,2 }\left( (0,1)^d, \mathbb{R} \right) }
\\  :=
  \Bigg(
    \int_{ (0,1)^d }
    \left| v( x ) \right|^2 dx
    +
    \int_{ (0,1)^d }
    \int_{ (0,1)^d }
    \frac{ \left| v( x ) - v( y ) \right|^2 }{ \left\| x - y \right\|^{ (d+ 2r ) }_{ \mathbb{R}^d } } \, dx \, dy
  \Bigg)^{ \frac{1}{2} } \in [0,\infty]
\end{multline*}
for all 
$ \mathcal{B}( (0,1)^d ) 
$/$
  \mathcal{B}( \mathbb{R} ) 
$-measurable functions 
$ 
  v \colon (0,1)^d \rightarrow \mathbb{R} 
$ 
and all $ r \in (0,1) $. 
Finally, we denote 
by $ v \cdot w : (0,1)^d \rightarrow \mathbb{R} $ the function
\[
  \left( v \cdot w \right)( x )
  =
  v( x ) \cdot w( x ),
  \qquad
  x \in (0,1)^d,
\]
for every $ v, w : (0,1)^d \rightarrow \mathbb{R} $.
Concerning the covariance operator 
of the Wiener process, let $ \mathcal{J} $ 
be a countable set, 
let $ ( g_j )_{ 
j \in \mathcal{J} } \subset U 
$ be an orthonormal basis of 
eigenfunctions of 
$ Q \colon U \rightarrow U $ 
and let 
$ ( \mu_j )_{ j \in \mathcal{J} } 
\subset [0,\infty) $ 
be the corresponding family 
of eigenvalues (such an 
orthonormal basis of 
eigenfunctions exists since 
$ Q \colon U \rightarrow U $ 
is a trace class operator, 
see Proposition~2.1.5 
in \cite{pr07}). 
In particular, we have
\[
  Q u
  =
  \sum_{ j \in \mathcal{J} }
  \mu_j
  \left< g_j, u \right>_U g_j
\]
for all $ u \in U $. Furthermore, we assume 
in this section that the 
eigenfunctions  $g_j \in U$, 
$ j \in \mathcal{J} $, are 
continuous and satisfy
\begin{equation}\label{ass_eigenf}
  \sup_{ j \in \mathcal{J} }
  \left\| g_j \right\|_{ C\left( (0,1)^d, \mathbb{R} \right) } < \infty
  \quad \text{and} \quad
  \sum_{ j \in \mathcal{J} }
  \left(
    \mu_j \left\| g_j \right\|^2_{ C^{\delta}\left( (0,1)^d, \mathbb{R} \right) }
  \right) < \infty
\end{equation}
for some $ \delta \in (0,1] $. We will give some concrete examples for $ ( g_j )_{ j \in \mathcal{J} } $ fulfilling \eqref{ass_eigenf} later. 

For {\bf the linear operator in 
Assumption \ref{semigroup}},
let $ \kappa \in (0,\infty) $ be a fixed real number, let $ \mathcal{I} = \mathbb{N}^d $ and let $ \lambda_i \in \mathbb{R} $, $ i \in \mathcal{I} $, and $ e_i \in H $, $ i \in \mathcal{I} $, be given by
\[
  \lambda_i 
  =
  \kappa
  \,
  \pi^2
  \big( 
    |i_1|^2 
    + \ldots 
    + |i_d|^2 
  \big),
\qquad
  e_i( x )
  =
  2^{ \frac{d}{2} }
  \sin( i_1 \pi x_1 )
  \cdot \ldots \cdot
  \sin( i_d \pi x_d ) 
\]
for all $ x \in ( x_1, \ldots, x_d ) \in (0,1)^d $ and all $ i = ( i_1, \ldots, i_d ) \in \mathbb{N}^d $. 
Next let $$ D(A) = \left\{ v \in H \colon
   \sum_{ i \in \mathcal{I} }
   \left|
     \lambda_i
   \right|^2
   \left| 
     \left< e_i, v \right>_H
   \right|^2 < \infty
 \right\} $$ and let
 $$
   A v = 
   \sum_{ i \in \mathcal{I} }
   - \lambda_i
   \left< e_i, v \right>_H
   e_i
 $$
 for all $ v \in D(A) $.
Hence, 
the linear 
operator $ A : D( A ) \subset H \rightarrow H $ 
in Assumption~\ref{semigroup} 
is nothing else but the Laplacian with Dirichlet boundary conditions 
times the constant $ \kappa \in (0,\infty) $, i.e. 
\begin{equation}
  A v
  =
  \kappa \cdot \Delta v
  =
  \kappa
  \left(
    \left( 
      \frac{ \partial^2 }{ \partial x_1^2 } 
    \right) v
    + \ldots +
    \left( 
      \frac{ \partial^2 }{ \partial x_d^2 } 
    \right) v
  \right)
\end{equation}
holds for all $ v \in D( A ) $ in this subsection 
(see, for instance, 
Subsection~3.8.1 in \cite{sy02}).

In view of 
{\bf the drift term in Assumption \ref{drift}}, 
let 
$ 
  f \colon (0,1)^d \times 
  \mathbb{R} \rightarrow \mathbb{R} 
$ 
be a 
$ 
  \mathcal{B}( 
    (0,1)^d \! \times \!
    \mathbb{R}
  )/\mathcal{B}( \mathbb{R} ) 
$-measurable function 
with 
$ 
  \int_{ (0,1)^d } 
  \left| f( x, 0 ) \right|^2 dx 
  < \infty 
$ 
and 
\begin{equation}
  \sup_{ x \in (0,1)^d } 
  \sup_{ 
    \substack{
      y_1, y_2 \in \mathbb{R}  
      \\
      y_1 \neq y_2
    }	    
  } 
  \left(
  \frac{
    \left| 
      f( x, y_1 ) -
      f( x, y_2 ) 
    \right| 
  }{
    \left| y_1 - y_2 \right|
  }
  \right)
  < \infty  .
\end{equation}
Then the (in general nonlinear) operator 
$ F \colon H \rightarrow H $ 
given by 
\begin{equation}
  \left( F( v ) \right)\!( x )
  =
  f( x, v( x ) ),
  \qquad
  x \in (0,1)^d,
\end{equation}
for all $v \in H $ satisfies 
Assumption~\ref{drift}, i.e. 
\begin{equation}
  \sup_{ 
    \substack{ v,w \in H \\ v \neq w } 
  }
  \frac{ 
    \left\| F( v ) - F( w ) \right\|_H 
  }{ 
    \left\| v - w \right\|_H 
  } < \infty
\end{equation}
holds. 

We now describe a class of 
{\bf diffusion terms satisfying Assumption \ref{diffusion}}. 
To this end 
let $ q \in [0,\infty) $
be a real number and let
$ 
  b \colon (0,1)^d \times 
  \mathbb{R} 
  \rightarrow \mathbb{R} 
$ 
be a function 
satisfying 
\begin{equation}\label{eq:qbound} 
  \left|
    b( x_1, y_1 ) - b( x_2, y_2 )
  \right|
  \leq
  q \,
  \big(
    \left\|
      x_1 - x_2
    \right\|_{ \mathbb{R}^d }
    +
    \left|
      y_1 - y_2
    \right|
  \big)
\end{equation}
for all $ x_1, x_2 \in (0,1)^d $
and all $ y_1, y_2 \in \mathbb{R} $.
In addition, we assume for simplicity
that 
$ 
  \int_{ (0,1)^d }
  \left| b( x, 0 ) \right|^2 
  dx
  \leq q^2
$.
Then let 
$ B \colon H \rightarrow HS( U_0, H ) $ 
be the (in general nonlinear) 
operator given by
\begin{equation}\label{ass_B}
  \Big( B( v ) u \Big)( x )
  =
  \Big( b( \cdot, v ) \cdot u \Big)( x )
  =
  b( x, v( x ) ) \cdot u( x ),
  \qquad
  x \in (0,1)^d,
\end{equation}
for all $ v \in H $ and all 
$ u \in U_0 \subset U $. 
We now check step by step 
that 
$ 
  B \colon H \rightarrow 
  HS( U_0, H ) 
$ 
given by \eqref{ass_B} 
satisfies 
Assumption~\ref{diffusion}. 
First of all, $ B $ is well defined. 
Indeed, we obviously have 
$ U_0 \subset L^{ \infty }( (0,1)^d , \mathbb{R} ) $ 
continuously due to \eqref{ass_eigenf}
and 
therefore, 
$ 
  B( v ) \colon U_0 \rightarrow H 
$ 
is a 
bounded linear operator from $ U_0 $ to $ H $ 
for every $ v \in H $. Moreover, we have
\begin{align*}
  \left\| B( v ) \right\|^2_{ HS( U_0, H ) }
 &=
  \sum_{ j \in \mathcal{ J } }
  \left\|
    B( v ) \sqrt{ \mu_j } g_j 
  \right\|_H^2
  =
  \sum_{ j \in \mathcal{ J } }
  \mu_j
  \left\|
    B( v ) g_j 
  \right\|_H^2
\\&=
  \sum_{ j \in \mathcal{ J } }
  \mu_j
  \left(
    \int_{ (0,1)^d }
    \left| b( x, v( x ) ) \cdot g_j( x ) \right|^2 dx
  \right)
\\&\leq 
  \sum_{ j \in \mathcal{ J } }
  \mu_j
  \left(
    \int_{ (0,1)^d }
    \left| b( x, v( x ) ) \right|^2 dx
  \right)
  \left(
    \sup_{ x \in (0,1)^d }
    \left| g_j( x ) \right|^2
  \right)
\end{align*}
and hence
\begin{align*}
  \left\|
    B( v ) 
  \right\|_{ HS( U_0, H ) }
 &\leq 
  \left\|
    b( \cdot, v )
  \right\|_H
  \left(
    \sum_{ j \in \mathcal{J} }
    \mu_j
  \right)^{ \frac{1}{2} }
  \left(
    \sup_{ j \in \mathcal{J} }
    \left\|
      g_j
    \right\|_{ C( (0,1)^d, \mathbb{R} ) }
  \right)
\\&=
  \left\|
    b( \cdot, v )
  \right\|_H
  \sqrt{ \text{Tr}( Q ) }
  \left(
    \sup_{ j \in \mathcal{J} }
    \left\|
      g_j
    \right\|_{ C( (0,1)^d, \mathbb{R} ) }
  \right)
  < \infty
\end{align*}
for all $ v \in H $ which shows 
that 
$ 
  B \colon H \rightarrow HS( U_0, H ) 
$ 
is well defined. 
Moreover, 
$ 
  B \colon H \rightarrow HS( U_0, H ) 
$ 
is globally Lipschitz continuous. 
More precisely, we have
\begin{align*}
 &\left\| B( v ) - B( w ) \right\|^2_{ HS( U_0, H ) }
  =
  \sum_{ j \in \mathcal{J} }
  \mu_j
  \left\|
    \left( B( v ) - B( w ) \right) g_j
  \right\|_H^2
\\&=
  \sum_{ j \in \mathcal{J} }
  \mu_j
  \left( 
    \int_{ (0,1)^d }
    \left| b( x, v( x ) ) - b( x, w( x ) ) \right|^2
    \left| g_j( x ) \right |^2 dx
  \right)
\\&\leq
  \left(
    \sum_{ j \in \mathcal{J} }
    \mu_j
  \right)
  \left( 
    \int_{ (0,1)^d }
      \left| b( x, v( x ) ) - b( x, w( x ) ) 
      \right|^2 
    dx
  \right)
  \left(
    \sup_{ j \in \mathcal{J} }
    \left\| g_j 
    \right\|_{ 
      C\left( (0,1)^d, \mathbb{R} \right) 
    }^2
  \right)
\end{align*}
and therefore
\begin{align*}
  \left\| B( v ) - B( w ) 
  \right\|_{ HS( U_0, H ) }
& \leq 
  q \left\| v - w \right\|_H
  \left( 
    \sum\nolimits_{ 
      j \in \mathcal{J} 
    } 
    \mu_j 
  \right)^{ \frac{1}{2} }
  \left(
    \sup_{ j \in \mathcal{J} }
    \left\| g_j \right\|_{ C\left( (0,1)^d, \mathbb{R} \right) }
  \right)
\\&=
  q \sqrt{ \text{Tr}( Q ) }
  \left(
    \sup_{ j \in \mathcal{J} }
    \left\| g_j \right\|_{ C\left( (0,1)^d, \mathbb{R} \right) }
  \right)
  \left\| v - w \right\|_H
\end{align*}
for all $ v, w \in H $. Hence, it remains 
to check
\begin{equation}\label{what_remains_to_check}
  B( V_{\alpha} )
  \subset 
  HS( U_0, V_{\alpha} )
  \quad \text{and} \quad 
  \left\| B( v ) \right\|_{ HS( U_0, V_{\alpha} ) }
  \leq 
  c \left( 1 + \left\| v \right\|_{ V_{\alpha} } \right)
\end{equation}
for every 
$ v \in V_{ \alpha } $ 
for appropriate 
$ \alpha \in [0, \frac{1}{2} ) $, $ c \in [0,\infty) $. 
In order to verify \eqref{what_remains_to_check}, 
several preparations are needed.
First, we review appropriate 
characterizations 
of the spaces 
$ 
  \left( 
    V_r, \left\| \cdot \right\|_{ V_r } 
  \right)
$, 
$ r \in (0,\frac{1}{2}) $,
from the literature. 
More formally, it is known that
\begin{equation}
\label{eq:vr1}
  V_r =
  \left\{ v \in H 
    \colon 
    \left\| v 
    \right\|_{ 
      W^{ 2r, 2 }( (0,1)^d, \mathbb{R} ) 
    } 
    < \infty
  \right\}
\end{equation}
holds for all $ r \in (0,\frac{1}{4}) $, that
\begin{equation}
\label{eq:vr2}
  V_r =
  \left\{ v \in H 
    \colon 
    \left\| v 
    \right\|_{ 
      W^{ 2r, 2 }( (0,1)^d, \mathbb{R} ) 
    } < \infty, \,
    v\big|_{ \partial(0,1)^d } \equiv 0
  \right\}
\end{equation}
holds for all $ r \in (\frac{1}{4}, \frac{1}{2}) $ and that there are real numbers $ C_r \in [1,\infty) $, $ r \in ( 0,\frac{1}{2} ) $, such that
\begin{equation}\label{eq:real_constant}
  \frac{1}{C_r}
  \left\| v \right\|_{ W^{ 2r, 2 }( (0,1)^d, \mathbb{R} ) }
  \leq 
  \left\| v \right\|_{ V_r }
  \leq 
  C_r
  \left\| v \right\|_{ W^{ 2r, 2 }( (0,1)^d, \mathbb{R} ) }
\end{equation}
holds for all $ v \in V_r $ and all $ r \in (0,\frac{1}{2} ) $ (see, e.g., A. Lunardi \cite{l09} or also (A.46) in \cite{dz92}). In particular, \eqref{eq:vr1} shows
\begin{equation}\label{eq:follows}
  \left\| v \right\|_{ W^{ 2r, 2 }( (0,1)^d, \mathbb{R} ) }
  < \infty
  \quad\qquad
  \Longrightarrow
  \quad\qquad
  v \in V_r
\end{equation}
for all $ \mathcal{B}( (0,1)^d ) $/$ 
\mathcal{B}( \mathbb{R} ) 
$-measurable functions 
$ v \colon (0,1)^d \rightarrow 
\mathbb{R} $ and 
all $ r \in (0,\frac{1}{4}) $. 
We remark that \eqref{eq:follows} 
does not hold for 
all $ r \in (\frac{1}{4}, \frac{1}{2}) $ 
instead of $ r \in (0, \frac{1}{4}) $ 
since a $ \mathcal{B}( (0,1)^d ) 
$/$ \mathcal{B}( \mathbb{R} ) 
$-measurable function 
$ 
  v \colon (0,1)^d \rightarrow 
  \mathbb{R} 
$ 
with 
$ 
  \left\| v 
  \right\|_{ 
    W^{ 2r, 2 }( (0,1)^d, \mathbb{R} ) 
  } < \infty 
$ 
for some 
$ r \in (0, \frac{1}{2}) $ 
does, in general, not 
fulfill the Dirichlet boundary 
conditions in \eqref{eq:vr2}.
In the next step observe that
\begin{align*}
  &
  \left\|
    v \cdot w
  \right\|^2_{ 
    W^{ r, 2 }( (0,1)^d, \mathbb{R} ) 
  }
\\&\leq 
  \int_{ (0,1)^d }
    \left| v(x) \cdot w( x ) \right|^2 
  dx
  +
  \int_{ (0,1)^d }
  \int_{ (0,1)^d }
  \frac{
    \left| 
      v(x) \cdot w(x)
      -
      v(y) \cdot w( y )
    \right|^2
  }{
    \left\| x - y \right\|^{ ( d + 2 r ) }_{ \mathbb{R}^d }
  } \, dx \, dy
\\&\leq 
  \left\| v \right\|^2_H
  \left\| w \right\|^2_{ 
    C( (0,1)^d, \mathbb{R} ) 
  }
  +
  2 \left\| w \right\|^2_{ 
    C( (0,1)^d, \mathbb{R} ) 
  }
  \int_{ (0,1)^d }
  \int_{ (0,1)^d }
  \frac{
    \left| 
      v(x) - v(y)
    \right|^2
  }{
    \left\| x - y  
    \right\|^{ 
      ( d + 2 r ) 
    }_{ \mathbb{R}^d }
  } \, dx \, dy
\\&\quad+
  2 \int_{ (0,1)^d }
  \int_{ (0,1)^d }
  \frac{
    \left| v(y) \right|^2
    \left| w( x ) - w( y ) \right|^2
  }{
    \left\| x - y \right\|^{ ( d + 2 r ) 
    }_{ \mathbb{R}^d }
  } \, dx \, dy
\\ & \leq 
  2 
  \left\| v 
  \right\|^2_{ 
    W^{ r, 2 }( (0,1)^d, \mathbb{R} ) 
  }
  \left\| w 
  \right\|^2_{ 
    C( (0,1)^d, \mathbb{R} ) 
  }
\\&\quad+
  2 
  \left\| v 
  \right\|^2_H
  \left(
    \sup_{ 
      \substack{
        x, y \in (0,1)^d
      \\
        x \neq y
      }
    }
    \frac{
      \left| w(x) - w(y) \right|^2
    }{
      \left\| x - y \right\|_{ 
        \mathbb{R}^d 
      }^{ 2 \delta }
    }
  \right)
  \left(
    \int_{ (-1,1)^d }
    \left\| y 
    \right\|^{ ( 2 \delta - d - 2 r ) } dy
  \right)
\end{align*}
for all $ v \in H $,
$ r \in (0, 1 ) $
and all
$ 
  \mathcal{B}(
    (0,1)^d
  )
$/$
  \mathcal{B}( \mathbb{R} )
$-measurable functions 
$ w \colon (0,1)^d \rightarrow \mathbb{R} 
$. The estimate
\begin{align}
\label{eq:intest}
&
  \int_{ (-1,1)^d }
  \left\| x  
  \right\|^{ z 
  }_{ \mathbb{R}^d }
  dx
\leq
  \int_{ 
    \left\{
      y \in \mathbb{R}^d 
      \colon
      \left\| y \right\|_{ \mathbb{R}^d }
      \leq \sqrt{d}
    \right\}
  } \!
  \left\| x  
  \right\|^{ z 
  }_{ \mathbb{R}^d }
  dx
\\ & =
\nonumber
  \frac{ 
    \pi^{ \frac{d}{2} }
    d 
  }{
    \Gamma( \frac{d}{2} + 1 
    )
  }
  \int_0^{ \sqrt{d} } 
    r^{ \left( z + d - 1 
    \right) } \,
  dr
\leq
  3^d
  \int_0^{ \sqrt{d} }
    r^{ \left( z + d - 1 
    \right) } \,
  dr
=
  \frac{
    3^d \,
    d^{ \frac{ ( z + d ) }{ 2 } }
  }{
    \left( z + d 
    \right)
  }
  \leq
  \frac{
    ( 3 d )^{ d }
  }{
    \left( d + z 
    \right)
  }
\end{align}
for all 
$ z \in (-d,d) $
therefore gives
\begin{align}
\label{eq:multvw}
&
  \left\|
    v \cdot w
  \right\|_{ 
    W^{ r, 2 }( (0,1)^d, \mathbb{R} ) 
  }
\nonumber
\\ & \leq 
\nonumber
  \sqrt{ 2 } 
  \left\| v 
  \right\|_{ 
    W^{ r, 2 }( (0,1)^d, \mathbb{R} ) 
  }
  \left(
    \left\| w 
    \right\|_{ 
      C( (0,1)^d, \mathbb{R} ) 
    }
    + 
    \sup_{ 
      \substack{
        x, y \in (0,1)^d
      \\
        x \neq y
      }
    }
    \frac{
      \left| w(x) - w(y) \right|
    }{
      \left\| x - y \right\|_{ 
        \mathbb{R}^d 
      }^{ \delta }
    }
    \cdot
    \frac{
      \left( 3 d \right)^{ \frac{ d }{ 2 } }
    }{ 
      \sqrt{ 2 \delta - 2 r } 
    }
  \right)
\\ & \leq 
  \left(
    \frac{
      \left( 3 d \right)^{ \frac{ d }{ 2 } }
    }{ 
      \sqrt{ \delta - r } 
    }
  \right)
  \left\| v 
  \right\|_{ 
    W^{ r, 2 }( (0,1)^d, \mathbb{R} ) 
  }
  \left\| w 
  \right\|_{ 
    C^{ \delta }( (0,1)^d, \mathbb{R} ) 
  }
\end{align}
for all 
$ 
  \mathcal{B}(
    (0,1)^d
  )
$/$
  \mathcal{B}( \mathbb{R} )
$-measurable functions 
$ v, w \colon (0,1)^d \rightarrow \mathbb{R} 
$ and
all $ r \in (0, \delta ) $
(see also Section~4 in 
H.~Triebel~\cite{t92}).
In addition, note that
the estimate
$ 
  \left( a + b \right)^2
  \leq
  2 a^2 + 2 b^2
$
for all 
$ a, b \in \mathbb{R} $
and inequality~\eqref{eq:qbound}
imply
\begin{align*}
  &\left\| b( \cdot, v ) 
  \right\|^2_{ 
    W^{ r, 2 }( (0,1)^d, \mathbb{R} ) 
  }
\\&=
  \int_{ (0,1)^d }
  \left| b( x, v( x ) ) \right|^2 dx
  +
  \int_{ (0,1)^d } 
  \int_{ (0,1)^d } 
  \frac{ 
    \left| b( x, v(x) ) - 
      b( y, v(y) ) 
  \right|^2 }{
    \left\| x - y 
    \right\|_{ \mathbb{R}^d 
    }^{ (d + 2 r ) } }
  \, dx \, dy
\\&\leq 
  \int_{ (0,1)^d } \!
  \left(
    q \left| v( x ) \right|
    +
    \left| b( x, 0 ) \right|
  \right)^2 dx 
  +
  2 \! \int_{ (0,1)^d } \!
  \int_{ (0,1)^d } \!\!\!
  \frac{ \left| b( x, v(x) ) - b( x, v(y) ) \right|^2 }{
    \left\| x - y \right\|_{ \mathbb{R}^d }^{ (d + 2 r ) } }
  dx \, dy
\\&\quad+
  2 \int_{ (0,1)^d }
  \int_{ (0,1)^d }
  \frac{ \left| b( x, v(y) ) - b( y, v(y) ) \right|^2 }{
    \left\| x - y \right\|_{ \mathbb{R}^d }^{ (d + 2 r ) } }
  \, dx \, dy
\\ & \leq
  2 \, q^2 
  \left\| 
    v 
  \right\|^2_{ 
    W^{ r, 2 }( (0,1)^d, \mathbb{R} ) 
  } 
  +
  2 \, q^2
  +
  2 \, q^2
  \int_{ (0,1)^d }
  \int_{ (0,1)^d }
  \left\| x - y 
  \right\|^{ (2-d-2r) }_{ \mathbb{R}^d }
  dx \, dy
\end{align*}
for all 
$ 
  \mathcal{B}(
    (0,1)^d
  )
$/$
  \mathcal{B}( \mathbb{R} )
$-measurable functions 
$ v \colon (0,1)^d \rightarrow 
\mathbb{R} $ and 
all $ r \in (0,1) $. 
Inequality~\eqref{eq:intest}
therefore shows
\begin{align*}
&
  \left\| b( \cdot, v ) 
  \right\|^2_{ 
    W^{ r, 2 }( (0,1)^d, \mathbb{R} ) 
  }
\leq 
  2 \, q^2 
  \left\| v 
  \right\|^2_{ 
    W^{ r, 2 }( (0,1)^d, \mathbb{R} ) 
  } 
  +
  2 \, q^2
  +
  q^2
  \frac{ 
    \left( 3d \right)^d 
  }{ 
    ( 1 - r ) 
  }
\\ & \leq 
  q^2 
  \left(
    2 
    \left\| 
      v 
    \right\|^2_{ 
      W^{ r, 2 }( (0,1)^d, \mathbb{R} ) 
    } 
    +
    \frac{ 
      2 \left( 3d \right)^d 
    }{ 
      ( 1 - r ) 
    }
  \right)
\leq 
  \left(
    \frac{ 
      q^2 \,
      2 \left( 3d \right)^d 
    }{ 
      ( 1 - r ) 
    }
  \right)
  \left(
    \left\| 
      v 
    \right\|^2_{ 
      W^{ r, 2 }( (0,1)^d, \mathbb{R} ) 
    } 
    +
    1
  \right)
\end{align*}
this finally yields
\begin{equation}\label{eq:klein_b}
  \left\| 
    b( \cdot, v ) 
  \right\|_{ 
    W^{ r, 2 }( (0,1)^d, \mathbb{R} ) 
  }
  \leq 
  \left(
    \frac{ q \left( 3 d \right)^d 
    }{ (1 - r) }
  \right)
  \left( 
    1
    +
    \left\| v \right\|_{ W^{ r, 2 }( (0,1)^d, \mathbb{R} ) }
  \right)
\end{equation}
for all 
$ 
  \mathcal{B}(
    (0,1)^d
  )
$/$
  \mathcal{B}( \mathbb{R} )
$-measurable functions 
$ v \colon (0,1)^d \rightarrow 
\mathbb{R} $ and 
all $ r \in (0,1) $. 
Combining
\eqref{eq:real_constant} 
and
\eqref{eq:klein_b}
then, in particular, shows
\begin{equation}\label{eq:klein_b2}
  \left\| b( \cdot, v ) \right\|_{ W^{ 2 r, 2 }( (0,1)^d, \mathbb{R} ) }
  \leq 
  \left(
    \frac{ q \, C_r \left( 3 d \right)^d 
    }{ (1 - 2r) }
  \right)
  \left( 
    1
    +
    \left\| v \right\|_{ V_r }
  \right)
  < \infty
\end{equation}
for all $ v \in V_r $ and all $ r \in (0,\frac{1}{2}) $.
Next we 
combine
\eqref{eq:multvw},
\eqref{eq:klein_b2}
and
\eqref{ass_eigenf}
to obtain
\begin{align}
\label{eq:boundHS}
&
  \left(
    \sum_{ j \in \mathcal{J} }
    \mu_j
    \left\| B(v) g_j \right\|_{
      W^{ 2 r, 2 }( (0,1)^d, \mathbb{R} )
    }^2
  \right)^{ \! \frac{ 1 }{ 2 } }
\nonumber
\\ & \leq
  \left(
    \frac{
      \left( 3 d \right)^{  
        \frac{ d }{ 2 }
      }
    }{ 
      \sqrt{ \delta - 2 r } 
    }
  \right)
  \left(
    \sum_{ j \in \mathcal{J} }
    \mu_j
    \left\| 
      g_j
    \right\|_{ 
      C^{ \delta }( (0,1)^d, \mathbb{R} ) 
    }^2
  \right)^{ \! \frac{ 1 }{ 2 } }
  \left\| 
    b( \cdot, v ) 
  \right\|_{ 
    W^{ 2 r, 2 }( (0,1)^d, \mathbb{R} ) 
  }
\\ & \leq
\nonumber
  \left(
    \frac{
      q \, C_r 
      \left( 3 d \right)^{  
        2 d
      }
    }{ 
      \left( 
        \delta - 2 r 
      \right)^{ 2 } 
    }
  \right)
  \left(
    \sum_{ j \in \mathcal{J} }
    \mu_j
    \left\| 
      g_j
    \right\|_{ 
      C^{ \delta }( (0,1)^d, \mathbb{R} ) 
    }^2
  \right)^{ \! \frac{ 1 }{ 2 } }
  \left(
    1 + \left\| v \right\|_{ V_r }
  \right)
< \infty
\end{align}
for all 
$ v \in V_r $
and all
$ r \in (0, \frac{\delta}{2}) $.
The Cauchy-Schwartz 
inequality and
estimate~\eqref{eq:boundHS}
then imply
\begin{align}
\label{eq:Bbound2}
&
\nonumber
  \left\| B(v) u \right\|_{
    W^{ 2 r, 2 }( (0,1)^d, \mathbb{R} )
  }
=
  \left\| 
    B(v) 
    \left( 
      \sum\nolimits_{ 
        j \in \tilde{\mathcal{J}} 
      }
      \mu_j 
      \left< g_j, u \right>_{ U_0 }
      g_j
    \right)
  \right\|_{
    W^{ 2 r, 2 }( (0,1)^d, \mathbb{R} )
  }
\\ & \leq
  \sum_{ j \in \tilde{\mathcal{J}} }
  \left(
  \mu_j \,
  \big| \!
  \left< g_j, u 
  \right>_{ U_0 }
  \! \big|
  \left\| B(v) g_j \right\|_{
    W^{ 2 r, 2 }( (0,1)^d, \mathbb{R} )
  }
  \right)
\nonumber
\\ & \leq
  \left(
  \sum_{ j \in \tilde{\mathcal{J}} }
  \big| \!
  \left< \sqrt{ \mu_j } g_j, u 
  \right>_{ U_0 }
  \! \big|^2
  \right)^{ \! \frac{ 1 }{ 2 } }
  \left(
    \sum_{ j \in \tilde{\mathcal{J}} }
    \mu_j
    \left\| B(v) g_j \right\|_{
      W^{ 2 r, 2 }( (0,1)^d, \mathbb{R} )
    }^2
  \right)^{ \! \frac{ 1 }{ 2 } }
\\ & \leq
\nonumber
  \left(
    \frac{
      q \, C_r 
      \left( 3 d \right)^{  
        2 d
      }
    }{ 
      \left( 
        \delta - 2 r 
      \right)^{ 2 } 
    }
  \right)
  \left(
    \sum_{ j \in \mathcal{J} }
    \mu_j
    \left\| 
      g_j
    \right\|_{ 
      C^{ \delta }( (0,1)^d, \mathbb{R} ) 
    }^2
  \right)^{ \! \frac{ 1 }{ 2 } }
  \left(
    1 + \left\| v \right\|_{ V_r }
  \right)
  \left\| u \right\|_{ U_0 }
< \infty
\end{align}
for all $ v \in V_{ r } $,
$ r \in (0, \frac{ \delta }{ 2 } ) $,
$ u \in U_0 $
with 
$ 
  u = \sum_{ j \in \tilde{ \mathcal{J} } 
  }
  \mu_j 
  \left< g_j, u \right>_{ U_0 }
  g_j
$ 
and all
finite subsets $ \tilde{\mathcal{J}}
\subset \mathcal{J} $ of
$ \mathcal{J} $.
This and \eqref{eq:follows}
then show
that 
$
  B(v) u \in V_{ r }
$
and that
\begin{equation}
\label{eq:Bbound}
  \left\| B(v) u \right\|_{ V_{ r } }
\leq
  \left(
    \frac{
      q \left| C_r \right|^2
      \left( 3 d \right)^{  
        2 d
      }
    }{ 
      \left( 
        \delta - 2 r 
      \right)^{ 2 } 
    }
  \right)
  \left(
    \sum_{ j \in \mathcal{J} }
    \mu_j
    \left\| 
      g_j
    \right\|_{ 
      C^{ \delta }( (0,1)^d, \mathbb{R} ) 
    }^2
  \right)^{ \! \frac{ 1 }{ 2 } }
  \left(
    1 + \left\| v \right\|_{ V_r }
  \right)
  \left\| u \right\|_{ U_0 }
\end{equation}
for all $ v \in V_{ r } $,
$ r \in (0, \min( \frac{1}{4}, 
\frac{ \delta }{ 2 } ) ) $
and all $ u \in U_0 $.
Therefore,
we obtain that
$ B(v) \in L(U_0, V_{ r } ) $
for all 
$ v \in V_{ r } $ and
all
$ r \in (0, \min( \frac{1}{4},
\frac{ \delta }{ 2 } ) ) $.
Hence, 
\eqref{eq:real_constant}
and
\eqref{eq:boundHS}
give
\begin{equation}
\begin{split}
\lefteqn{
  \sum_{ j \in \mathcal{J} }
  \left(
    \mu_j
    \left\| 
      B(v) g_j
    \right\|_{
      V_{ r }
    }^2  
  \right)
}
\\ & \leq
  \left(
    \frac{
      q \left| C_r \right|^2
      \left( 3 d \right)^{  
        2 d
      }
    }{ 
      \left( 
        \delta - 2 r 
      \right)^{ 2 } 
    }
  \right)^{ \! 2 }
  \left(
    \sum_{ j \in \mathcal{J} }
    \mu_j
    \left\| 
      g_j
    \right\|_{ 
      C^{ \delta }( (0,1)^d, \mathbb{R} ) 
    }^2
  \right)
  \left(
    1 + \left\| v \right\|_{ V_r }
  \right)^2
< \infty
\end{split}
\end{equation}
for all $ v \in V_r $
and all
$ r \in 
(0, \min( \frac{1}{4}, \frac{ \delta }{ 2 } ) 
$.
Therefore, we obtain
$
  B(v) \in HS( U_0, V_r )
$
and
\begin{equation}
\label{eq:check6a}
\begin{split}
\lefteqn{
  \left\|
    B(v)
  \right\|_{
    HS( U_0, V_r ) 
  }
}
\\ & \leq
  \left(
    \frac{
      q \left| C_r \right|^2
      \left( 3 d \right)^{  
        2 d
      }
    }{ 
      \left( 
        \delta - 2 r 
      \right)^{ 2 } 
    }
  \right)
  \left(
    \sum_{ j \in \mathcal{J} }
    \mu_j
    \left\| 
      g_j
    \right\|_{ 
      C^{ \delta }( (0,1)^d, \mathbb{R} ) 
    }^2
  \right)^{ \! \frac{ 1 }{ 2 } }
  \left(
    1 + \left\| v \right\|_{ V_r }
  \right)
< 
  \infty
\end{split}
\end{equation}
for all 
$
  v \in V_r
$
and all
$
  r \in (0, \min( \frac{1}{4}, 
  \frac{ \delta }{ 2 } ) )
$.
This finally shows that 
Assumption \ref{diffusion} 
is fulfilled for
all $ \alpha \in [ 0, \min(\frac{1}{4}, \frac{\delta}{2} ) ) $. 

Concerning {\bf the initial value 
in Assumption \ref{initial}},
let 
$ x_0 \colon [0,1]^d \rightarrow 
\mathbb{R} $ be a 
twice continuously 
differentiable 
function 
with $ x_0|_{ \partial (0,1)^d } \equiv 0 $.
Then the 
$ \mathcal{F}_0 
$/$ \mathcal{B}( V_{\gamma} ) 
$-measurable 
mapping 
$ \xi \colon \Omega \rightarrow 
V_{ \gamma } $
given by $ \xi( \omega ) = x_0 $ for all $ \omega \in \Omega $
fulfills Assumption \ref{initial} for all $ \gamma \in [\alpha, \frac{1}{2} + \alpha ) $ and all $ p \in [2, \infty) $.

Having constructed examples 
of Assumptions~\ref{semigroup}-\ref{initial}, we now formulate the SPDE
\eqref{eq:SPDE} in the setting of this section. More formally,
under the setting above the SPDE \eqref{eq:SPDE} reduces to
\begin{equation}\label{eq:SPDE_reduced}
  d X_t( x )
  =
  \Big[
    \kappa
    \Delta
    X_t( x )
    +
    f( x, X_t( x ) )
  \Big] dt
  +
  b( x, X_t( x ) ) \, dW_t( x )
\end{equation}
with 
$ 
  X_{ t } 
  |_{ \partial( 0,1 )^d } 
  \equiv 0 
$
and $ X_0( x ) = x_0( x ) $ for $ t \in [0,T] $ and $ x \in (0,1)^d $.
Moreover, we define a family 
$ \beta^j \colon 
[0,T] \times \Omega 
\rightarrow \mathbb{R} $, $ j \in \{ k \in \mathcal{J} \, \big| \,
\mu_k \neq 0 \} $, of independent standard Brownian motions by
\[
  \beta_t^j( \omega )
  :=
  \frac{ 1 }{ \sqrt{ \mu_j } }
  \left< g_j, W_t( \omega ) \right>_U
\]
for all $ \omega \in \Omega $, $ t \in [0,T] $ and all
$ j \in \mathcal{J} $ with $ \mu_j \neq 0 $. Using this notation, 
the SPDE \eqref{eq:SPDE_reduced} 
can be written as
\begin{equation}
\label{eq:SPDE19}
  d X_t( x )
  =
  \Big[
    \kappa
    \Delta
    X_t( x )
    +
    f( x, X_t( x ) )
  \Big] dt
  +
  \sum_{ 
    \substack{ j \in \mathcal{J} \\ 
    \mu_j \neq 0 } 
  }
  \Big[
    \sqrt{ \mu_j } \,
    b( x, X_t( x ) ) \,
    g_j( x )
  \Big] d\beta_t^j
\end{equation}
with 
$ 
  X_{ t } 
  |_{ \partial( 0,1 )^d } 
  \equiv 0 
$
and $ X_0( x ) = x_0( x ) $ for $ t \in [0,T] $ and $ x \in (0,1)^d $.
Finally, due to \eqref{eq:check6a}, 
Theorem~\ref{thm:existence} 
shows the existence of an 
up to modifications unique 
predictable stochastic process 
$ 
  X \colon [0,T] \times 
  \Omega \rightarrow V_{ \gamma } 
$
fulfilling \eqref{eq:SPDE19} for any $ \gamma \in [ 0, \frac{ \min( 3, 2\delta + 2 ) }{ 4 } ) $. 

At this point we would like to 
thank an anonymous referee
for pointing out to us 
%his instructive observation 
that Theorem~\ref{thm:existence}
can be generalized to SPDEs 
on UMD Banach spaces with 
type $ 2 $ by exploiting
the results in 
van Neerven 
et al.~\cite{vvw07}.
In such a Banach space framework
the state space 
$ L^q( (0,1)^2, \mathbb{R} ) $
with possibly large $ q \in [2,\infty) $
can be considered instead of the
Hilbert space
$
  H = L^2 ( (0,1)^d, \mathbb{R} )
$.
By using appropriate
Sobolev embeddings we then
expect that one can even show
that the solution process
of the SPDE~\eqref{eq:SPDE19}
enjoys values in the space
$ 
  C^{ 2 \gamma 
  }( (0,1)^d, \mathbb{R} )
$
of 
continuous differentiable functions
from $ (0,1)^d $ to $ \mathbb{R} $
with
$ (2 \gamma - 1) $-H\"{o}lder 
continuous derivatives
for any 
$ 
  \gamma \in \big( \frac{1}{2}, 
    \frac{
      \min( 3 , 2 \delta + 2 )
    }{ 4 }
  \big) 
$.
The precise
regularity study
of the SPDE~\eqref{eq:SPDE19}
in such a Banach space framework 
instead of the Hilbert space
framework considered here
remains an open question 
for future research.

In the next step 
we illustrate Theorem \ref{thm:existence} 
using \eqref{eq:Bbound2} and
\eqref{eq:check6a} in the following three more concrete examples.
\subsection{ A one dimensional stochastic 
reaction diffusion equation}\label{sec:example1}
Consider the situation
described above in the case $d=1$.
In this subsection we want to 
give a concrete example for
$ \left( g_j \right)_{
  j \in \mathcal{J}
} $
and
$ \left( \mu_j \right)_{
  j \in \mathcal{J}
} $
so that \eqref{ass_eigenf}
is fulfilled and
all above applies.
Let
$ \mathcal{J} = \{ 0,1,2,\ldots \} $, let $ g_0( x ) = 1 $ and let $ g_j( x ) = \sqrt{2} \cos( j \pi x ) $ 
for all $ x \in (0,1) $ and 
all $ j \in \mathbb{N} $. 
Moreover, let 
$ \rho \in (1,\infty) $ and
$ \nu \in (0,\infty) $ be given 
real numbers, 
let $ \mu_0 = 0 $ and 
let $ \mu_j = \frac{ \nu }{ j^\rho } $ 
for all $ j \in \mathbb{N} $. This choice ensures that
\eqref{ass_eigenf} 
is fulfilled. 
Indeed, we have
\begin{align*}
&
  \sum_{ j \in \mathcal{J} }
  \mu_j \left\| g_j 
  \right\|^2_{ C^{ \delta }( (0,1)^d, \mathbb{R} ) }
\\ & =
  \sum_{ j = 1 }^{ \infty }
  \frac{ \nu }{ j^{ \rho } }
  \left\| g_j 
  \right\|^2_{ C^{ \delta }( (0,1)^d, \mathbb{R} ) }
\\&=
  \sum_{ j = 1 }^{ \infty }
  \frac{ 2 \nu }{ j^{ \rho } }
  \left(
    1 + \sup_{ \substack{ x,y \in (0,1) \\ x \neq y } }
    \frac{ \left| \cos( j \pi x ) - \cos( j \pi y ) \right| 
    }{ \left| x - y \right|^{ \delta } }
  \right)^{2}
\\&\leq 
  \sum_{ j = 1 }^{ \infty }
  \frac{ 2 \nu }{ j^{ \rho } }
  \left(
    1 + 
    \sup_{ 
      \substack{ 
          x,y \in (0,1) 
        \\ 
          x \neq y 
      } 
    }
    \frac{ 2^{ (1-\delta) } \left| \cos( j \pi x ) - 
      \cos( j \pi y ) \right|^{ \delta } 
    }{ \left| x - y \right|^{ \delta } }
  \right)^{2}
\end{align*}
and hence
\begin{equation}
\label{eq:gj_estimate}
\begin{split}
\lefteqn{
  \sum_{ j \in \mathcal{J} }
  \mu_j \left\| g_j 
  \right\|^2_{ 
    C^{ \delta }( (0,1)^d, \mathbb{R} 
    ) 
  }
\leq 
  \sum_{ j = 1 }^{ \infty }
  \frac{ 2 \nu }{ j^{ \rho } }
  \left(
    1 + 2^{ (1 - \delta) } ( j \pi )^{ \delta }
  \right)^{ \! 2 }
}
\\ & \leq
  \sum_{ j = 1 }^{ \infty }
  \frac{ 2 \nu }{ j^{ \rho } }
  \Big(
    1 + \pi \, j^{ \delta }
  \Big)^{ \! 2 }
  \leq 
  8 \nu \pi^2
  \left( 
    \sum_{ j = 1 }^{ \infty } 
    j^{ (2\delta - \rho) }
  \right)
  < \infty
\end{split}
\end{equation}
for all 
$ 
  \delta \in 
  ( 0, \frac{ \rho - 1 }{ 2 } ) 
$. 
Assumption \ref{diffusion} 
is thus fulfilled for 
every 
$ \alpha \in (0,\min(\frac{1}{4}, 
\frac{ \rho - 1 }{4} ) ) $ 
$= ( 0, \frac{ \min( 1, \rho - 1 ) }{ 4 } ) $ 
(see \eqref{eq:check6a}).
Here the SPDE \eqref{eq:SPDE19} 
reduces to
\begin{equation}\label{eq:SPDE21}
  d X_t( x )
  = \!\!
  \left[
    \kappa
    \frac{ \partial^2 }{ \partial x^2 }
    X_t( x )
    +
    f( x, X_t( x ) )
  \right] \! dt
  + \!
  \sum_{ j = 1 }^{ \infty } \!
  \left[
    \frac{ \sqrt{ 2 \nu } }{ 
      j^{ \frac{ \rho }{ 2 } } 
    } \,
    b( x, X_t( x ) ) \,
    \cos( j \pi x ) 
  \right] \! d\beta_t^j
\end{equation}
with $ X_t( 0 ) = X_t( 1 ) = 0 $ and $ X_0( x ) = x_0( x ) $ for $ t \in [0,T] $ and $ x \in (0,1) $. Theorem \ref{thm:existence}
finally yields the existence 
of an up to modifications 
unique stochastic process 
$ X \colon [0,T] \times \Omega 
\rightarrow V_{ \gamma } $ 
fulfilling \eqref{eq:SPDE21} 
for any 
$ 
  \gamma \in 
  [0, \frac{ \min( 3, \rho + 1 ) }{ 4 } ) 
$. 
Under further assumptions on 
$ 
  b \colon (0,1) \times \mathbb{R} 
  \rightarrow \mathbb{R} 
$, the solution 
of \eqref{eq:SPDE21}
enjoys even more regularity 
which is demonstrated 
in the following subsection.
\subsection{More regularity for a one dimensional stochastic reaction diffusion equation}\label{sec:example2}
Consider the situation
of 
Subsection~\ref{sec:example1}
with $ \rho = 3 $.
% As before let $ d = 1 $, let $ \mathcal{J} = \{0,1,2,\ldots\} $, let
% $ g_0( x ) = 1 $ and let $ g_j( x ) = \sqrt{ 2 } \cos( j \pi x ) $ for
% all $ x \in (0,1) $ and all $ j \in \mathbb{N} $ in this subsection.
% Additionally, let $ \nu \in (0,\infty) $ be a given real number, let
% $ \mu_0 = 0 $ and let $ \mu_j = \frac{ \nu }{ j^3 } $ for all $ j \in \mathbb{N} $. 
Hence, \eqref{eq:gj_estimate} 
shows that
\eqref{ass_eigenf} holds for all $ \delta \in (0,1) $. Therefore,
\eqref{eq:check6a} gives that Assumption \ref{diffusion} is fulfilled
for all $ \alpha \in [0,\frac{1}{4} ) $. However, we now additionally
assume that the diffusion coefficient 
$ b\colon (0,1) \times \mathbb{R} \rightarrow \mathbb{R} $ respects the Dirichlet boundary conditions in
\eqref{eq:SPDE19}, 
i.e., 
we assume that
\begin{equation}
\label{eq:addi_assumpt}
  \lim_{ x \searrow 0 }
  b( x, x ) 
  =
  \lim_{ x \nearrow 1 }
  b( x, x - 1 )
  =
  0
\end{equation}
holds. 
Under this additional 
assumption more regularity 
for the solution process 
of \eqref{eq:SPDE19} can 
be established. 
More precisely, 
\eqref{eq:addi_assumpt},
\eqref{eq:vr2} and 
\eqref{eq:Bbound2} 
yield that 
$
  B(v) u \in V_{ r }
$
and that
\begin{equation}
  \left\| B(v) u \right\|_{ V_{ r } }
\leq
  \left(
    \frac{
      q \left| C_r \right|^2
      \left( 3 d \right)^{  
        2 d
      }
    }{ 
      \left( 
        \delta - 2 r 
      \right)^{ 2 } 
    }
  \right)
  \left(
    \sum_{ j \in \mathcal{J} }
    \mu_j
    \left\| 
      g_j
    \right\|_{ 
      C^{ \delta }( (0,1)^d, \mathbb{R} ) 
    }^2
  \right)^{ \! \frac{ 1 }{ 2 } }
  \left(
    1 + \left\| v \right\|_{ V_r }
  \right)
  \left\| u \right\|_{ U_0 }
\end{equation}
for all $ v \in V_{ r } $,
$ r \in (\frac{1}{4}, \frac{ 1 }{ 2 } ) $
and all $ u \in U_0 $.
This implies that
$ B(v) \in L(U_0, V_{ r } ) $
for all 
$ v \in V_{ r } $ and
all
$ r \in (\frac{1}{4}, \frac{ 1 }{ 2 } ) $.
Hence, 
\eqref{eq:real_constant}
and
\eqref{eq:boundHS}
give
\begin{equation}
\begin{split}
\lefteqn{
  \sum_{ j \in \mathcal{J} }
  \left(
    \mu_j
    \left\| 
      B(v) g_j
    \right\|_{
      V_{ r }
    }^2  
  \right)
}
\\ & \leq
  \left(
    \frac{
      q \left| C_r \right|^2
      \left( 3 d \right)^{  
        2 d
      }
    }{ 
      \left( 
        \delta - 2 r 
      \right)^{ 2 } 
    }
  \right)^{ \! 2 }
  \left(
    \sum_{ j \in \mathcal{J} }
    \mu_j
    \left\| 
      g_j
    \right\|_{ 
      C^{ \delta }( (0,1)^d, \mathbb{R} ) 
    }^2
  \right)
  \left(
    1 + \left\| v \right\|_{ V_r }
  \right)^2
< \infty
\end{split}
\end{equation}
for all $ v \in V_r $
and all
$ r \in 
(\frac{1}{4}, \frac{ 1 }{ 2 } ) 
$.
Thus, Assumption~\ref{diffusion}
is here even fulfilled for all 
$ \alpha \in [0, \frac{1}{2} ) $.
Theorem \ref{thm:existence} finally shows that, 
under 
condition \eqref{eq:addi_assumpt}, 
the SPDE
\begin{equation*}
  d X_t( x )
  = \!\!
  \left[
    \kappa
    \frac{ \partial^2 }{ \partial x^2 }
    X_t( x )
    +
    f( x, X_t( x ) )
  \right] \! dt
  + \!
  \sum_{ j=1 }^{ \infty } \!
  \left[
    \frac{ \sqrt{ 2 \nu } }{ j^{ \frac{3}{2} } } \,
    b( x, X_t( x ) ) \,
    \cos( j \pi x )
  \right] \! d\beta_t^j
\end{equation*}
with $ X_t( 0 ) = X_t( 1 ) = 0 $ and $ X_0( x ) = x_0( x ) $
for $ t \in [0,T] $ and $ x \in ( 0,1 ) $ admits an up to
modifications unique predictable 
solution 
process 
$ 
  X \colon 
  [0,T] \times \Omega \rightarrow 
  V_{ \gamma } 
$ 
for any 
$ \gamma \in [0,1) $.
\subsection{Stochastic reaction diffusion equations with commutative noise}
Consider the situation
before 
Subsection~\ref{sec:example1}
and assume that 
the eigenfunctions of the linear 
operator 
$
  A \colon D( A ) 
  \subset H \rightarrow H 
$ 
and of the
covariance operator 
$ 
  Q \colon 
  U = H \rightarrow H 
$ 
coincide.
More formally, 
let 
$ 
  \mathcal{J} = 
  \mathcal{I} = 
  \mathbb{N}^d 
$,
let 
$ g_j = e_j $ 
for all $ j \in \mathcal{J} $, 
let $ \rho \in (d, d + 2) $ 
and 
$ \nu \in (0,\infty) $ 
be given real numbers 
and let 
$ 
  \mu_j = \nu
  \left( 
    j_1 + \ldots + j_d 
  \right)^{ - \rho } 
$ 
for all 
$ j \in ( j_1, \ldots, j_d ) \in 
\mathcal{J} = \mathbb{N}^d $. 
We now check condition
\eqref{ass_eigenf}. 
To this end note that
\begin{align*}
&
  \left\|
    g_j'( x )
  \right\|_{ 
    L( \mathbb{R}^d, \mathbb{R} ) 
  }
=
  \sup_{ \substack{ v \in \mathbb{R}^d \\ \left\| v \right\|_{ \mathbb{R}^d } \leq 1 } }
  \left| g_j'( x ) v \right|
  \leq 
  \sup_{ 
    \substack{ 
      v \in \mathbb{R}^d 
    \\ 
      \left\| v \right\|_{ \mathbb{R}^d } 
      \leq 1 
    } 
  }
  \left(
    \sum_{ k = 1 }^{ d }
    \left|
      \left( 
        \frac{ \partial g_j
        }{ \partial x_k } 
      \right) \! ( x )
    \right|
    \cdot 
    \left| v_k \right|
  \right)
\\ & \leq
  \left(
    \sum_{ k = 1 }^{ d }
    \left|
      \left( 
        \frac{ \partial g_j
        }{ \partial x_k }  
      \right) \! ( x )
    \right|^2
  \right)^{ \! \frac{1}{2} }
  \leq
  \left(
    \sum_{ k = 1 }^{ d }
    \pi^2 
    \left| j_k \right|^2
    2^d
  \right)^{ \! \frac{1}{2} }
=
  2^{ \frac{d}{2} } \pi
  \left(
    \sum_{ k = 1 }^{ d }
    \left| j_k \right|^2
  \right)^{ \! \frac{1}{2} }
\end{align*}
holds for all $ x \in (0,1)^d $ and all 
$ j \in ( j_1, \ldots, j_d ) \in \mathcal{J} $.
This implies
\begin{equation}
\label{eq:g_j_difference}
\begin{split}
  \left|
    g_j( x ) - g_j( y )
  \right|
 &\leq
  \int_0^1
  \left|
    g_j'( x + r( y - x ) )( y - x )
  \right| dr
\\&\leq 
  2^{ \frac{d}{2} } \pi 
  \left(
    \sum_{ k = 1 }^{ d }
    \left| j_k \right|^2
  \right)^{ \! \frac{1}{2} }
  \left\| x - y \right\|_{ \mathbb{R}^d }
\end{split}
\end{equation}
for all $ x, y \in (0,1)^d $ 
and all $ j \in \mathcal{J} $.
Hence, we obtain
\begin{align*}
  \left\| 
    g_j 
  \right\|_{ C^{ \delta }( (0,1)^d, \mathbb{R} ) }
 &\leq 
  \left\| 
    g_j 
  \right\|_{ C( (0,1)^d, \mathbb{R} ) }
  +
  \sup_{ \substack{x,y \in (0,1)^d \\ x \neq y }}
  \frac{ 
    \left| g_j( x ) - g_j( y ) \right|
  }{
    \left\| x - y \right\|^{ \delta }_{ \mathbb{R}^d }
  }
\\&\leq 
  2^{ \frac{d}{2} }
  +
  \sup_{ \substack{x,y \in (0,1)^d \\ x \neq y }}
  \frac{ 
    ( 
      2 \cdot 2^{ \frac{ d }{ 2 } } 
    )^{ ( 1 - \delta ) }
    \left| g_j( x ) - g_j( y ) \right|^{ \delta }
  }{
    \left\| x - y \right\|^{ \delta }_{ \mathbb{R}^d }
  }
\end{align*}
and
\begin{equation}
\label{eq:g_j_norm}
\begin{split}
  \left\| 
    g_j 
  \right\|_{ C^{ \delta }( (0,1)^d, \mathbb{R} ) }
 & \leq 
  2^{ \frac{d}{2} }
  +
  2^{ ( \frac{d}{2} + 1 )( 1 - \delta ) }
  \left(
    2^{ \frac{d}{2} } \pi
    \left(
      \sum_{ k = 1 }^{ d }
      \left| j_k \right|^2
    \right)^{ \! \frac{1}{2} \; }
  \right)^{ \! \delta }
\\ & \leq
  2^{ \frac{d}{2} }
  +
  2^{ \frac{d}{2} } \pi
  \left(
    \sum_{ k = 1 }^{ d }
    \left| j_k \right|^2
  \right)^{ \! \frac{\delta}{2} }
  \leq 
  2^{ ( \frac{d}{2} + 1 ) }
  \,
  \pi
  \left(
    \sum_{ k = 1 }^{ d }
    \left| j_k \right|^2
  \right)^{ \! \frac{\delta}{2} }
\end{split}
\end{equation}
for all $ \delta \in (0,1] $ 
and all $ j \in \mathcal{J} $.
Therefore, we get
\begin{align*}
  \sum_{ j \in \mathcal{J} }
  \mu_j
  \left\| 
    g_j 
  \right\|^2_{ 
    C^{ \delta }( 
      (0,1)^d, \mathbb{R} 
    ) 
  }
& \leq 
  \sum_{ j \in \mathbb{N}^d }
  \nu 
  \left( j_1 + \ldots + j_d 
  \right)^{ - \rho }
  2^{ \left( d + 2 \right) } 
  \,
  \pi^2
  \left(
    \sum_{ k = 1 }^{ d }
    \left| j_k \right|^2
  \right)^{ \! \delta }
\\&=
  \nu \,
  2^{ \left( d + 2 \right) } 
  \,
  \pi^2
  \left(
    \sum_{ j \in \mathbb{N}^d }
    \frac{
      \left( 
        \left| j_1 \right|^2 
        + \ldots 
        + 
        \left| j_d \right|^2 
      \right)^{ \! \delta } 
    }{ 
      \left( 
        j_1 + \ldots + j_d 
      \right)^{ \delta } 
    }
  \right)
  < \infty
\end{align*}
for all $ \delta \in (0, \frac{ \rho - d }{ 2 } ) $ 
and hence,
\eqref{ass_eigenf} holds for all 
$ \delta \in (0, \frac{ \rho - d }{ 2 } ) $.
Furthermore, 
since 
$ 
  g_j|_{
    \partial (0,1)^d
  } 
  = 
  e_j|_{
    \partial (0,1)^d
  }
  =
  0
$ 
for all $ j \in \mathcal{J} $ 
here, 
\eqref{eq:vr1},
\eqref{eq:vr2} and 
\eqref{eq:Bbound2} 
yield that 
$
  B(v) u \in V_{ r }
$
and that
\begin{equation}
  \left\| B(v) u \right\|_{ V_{ r } }
\leq
  \left(
    \frac{
      q \left| C_r \right|^2
      \left( 3 d \right)^{  
        2 d
      }
    }{ 
      \left( 
        \delta - 2 r 
      \right)^{ 2 } 
    }
  \right)
  \left(
    \sum_{ j \in \mathcal{J} }
    \mu_j
    \left\| 
      g_j
    \right\|_{ 
      C^{ \delta }( (0,1)^d, \mathbb{R} ) 
    }^2
  \right)^{ \! \frac{ 1 }{ 2 } }
  \left(
    1 + \left\| v \right\|_{ V_r }
  \right)
  \left\| u \right\|_{ U_0 }
\end{equation}
for all $ v \in V_{ r } $,
$ 
  r \in ( 0, \frac{ \rho - d }{ 4 } )
  \backslash \{ \frac{ 1 }{ 4 } \} 
$
and all $ u \in U_0 $.
This implies that
$ B(v) \in L(U_0, V_{ r } ) $
for all 
$ v \in V_{ r } $ and
all
$ 
  r \in ( 0, \frac{ \rho - d }{ 4 } ) 
  \backslash \{ \frac{ 1 }{ 4 } \} 
$.
Hence, 
\eqref{eq:real_constant}
and
\eqref{eq:boundHS}
give
\begin{equation}
\begin{split}
\lefteqn{
  \sum_{ j \in \mathcal{J} }
  \left(
    \mu_j
    \left\| 
      B(v) g_j
    \right\|_{
      V_{ r }
    }^2  
  \right)
}
\\ & \leq
  \left(
    \frac{
      q \left| C_r \right|^2
      \left( 3 d \right)^{  
        2 d
      }
    }{ 
      \left( 
        \delta - 2 r 
      \right)^{ 2 } 
    }
  \right)^{ \! 2 }
  \left(
    \sum_{ j \in \mathcal{J} }
    \mu_j
    \left\| 
      g_j
    \right\|_{ 
      C^{ \delta }( (0,1)^d, \mathbb{R} ) 
    }^2
  \right)
  \left(
    1 + \left\| v \right\|_{ V_r }
  \right)^2
< \infty
\end{split}
\end{equation}
for all $ v \in V_r $
and all
$ 
  r \in ( 0, \frac{ \rho - d }{ 4 } ) 
  \backslash \{ \frac{ 1 }{ 4 } \} 
$.
Assumption~\ref{diffusion} 
is thus fulfilled for all
$ 
  \alpha \in [0,\frac{\rho-d}{4})
  \backslash \{ \frac{ 1 }{ 4 } \} 
$ here. 
Theorem~\ref{thm:existence}
therefore yields that 
the SPDE
\begin{multline}\label{eq:SPDE_Z}
  d X_t( x )
  =
  \Big[
    \kappa
    \Delta
    X_t( x )
    +
    f( x, X_t( x ) )
  \Big] dt
\\
  +
  \sum_{ j \in \mathbb{N}^d }
  \left[
    \frac{ 
      \sqrt{ \nu 2^d } 
      \sin( j_1 \pi x_1 ) \cdot \ldots \cdot
      \sin( j_d \pi x_d ) 
    }{ 
      \left( 
        j_1 + \ldots + j_d
      \right)^{ 
        \frac{\rho}{2} } 
      } \,
    b( x, X_t( x ) )
  \right] \! d\beta_t^j
\end{multline}
with 
$ 
  X_{ t } |_{ \partial( 0,1 )^d } 
  \equiv 0 
$ 
and $ X_0( x ) = x_0( x ) $ 
for all $ t \in [0,T] $ and $ x \in (0,1)^d $ enjoys an
up to modifications unique 
predictable solution process
$ 
  X \colon [0,T] \times \Omega 
  \rightarrow V_{ \gamma } 
$
fulfilling \eqref{eq:SPDE_Z} for any
$ \gamma \in [0, \frac{ \rho-d + 2 }{ 4 } ) $.
\section{Proof of 
Theorem \ref{thm:existence}}\label{secproofs}
Throughout this section the notation
$$
  \left\|
    Z
  \right\|_{L^p (\Omega; E)}
  :=
  \Big(
    \mathbb{E}\big[
      \left\|
        Z
      \right\|_E^p
    \big]
  \Big)^{ \! \frac{1}{p} }
  \in [0,\infty]
$$
is used for an 
$ \mathbb{R} $-Banach
space 
$ \left( E, \left\| \cdot \right\|_E \right) $
and an 
$ \mathcal{F} $/$ \mathcal{B}(E) $-measurable
mapping 
$ Z \colon \Omega \rightarrow E $.
The real number $ p \in [2,\infty) $ is as
given in Assumption \ref{initial}.
Next a 
well known estimate
for analytic semigroups
is presented
(see, e.g.,
Lemma~11.36 in 
Renardy and Roggers~\cite{rr93})
%and also 
%Theorem~37.5 in \cite{sy02}).
%
%
%
%
\begin{lemma}\label{lemmaRef4}
Assume that the setting
in Section~\ref{sec:setting} is fulfilled. 
Then there exist real numbers
$ c_r \in [1,\infty ) $,
$ r \in [0,1] $, such that
\begin{equation}
  \left\| 
    \left( t \left( \eta - A \right) \right)^r
    e^{ A t }
  \right\|_{ L(H) } 
  \leq c_r
\end{equation}
and
\begin{equation}
  \left\|
    \left( 
      t \left( \eta - A \right) 
    \right)^{ - r }
    \left( e^{ A t } - I 
    \right)
  \right\|_{ L ( H ) }
\leq
  c_r
\end{equation}
for all $ t \in (0,T] $
and all $ r \in [0,1] $.  
\end{lemma}
Moreover, we would like to note
the following remark.
\begin{remark}\label{rem:analytic}
Assume that the setting
in Section~\ref{sec:setting} 
is fulfilled and let 
$ Y \colon [0,T] \times \Omega
\rightarrow HS(U_0, H) $ be
a predictable stochastic process.
Then we obtain
$ 
  e^{ A t } \, Y_s(\omega) 
  \in 
  \cap_{ u \in [0,\infty) }
  \, V_u 
$ 
for all 
$ \omega \in \Omega,
s \in [0,T] $ and all $ t \in (0,T] $
since the semigroup 
$ e^{ A t } \in L(H) $,
$ t \in [0,\infty) $,
is analytic
(see Assumption~\ref{semigroup}).
In particular, if
$ 
  \int^t_0
  \mathbb{E}\big[ 
    \| 
      e^{ A (t-s) }
      \, Y_s 
    \|_{ HS(U_0, V_r) }^2
  \big] \, ds 
  < \infty 
$ 
for all $ t \in [0,T] $
and some $ r \in [0,\infty) $,
then the stochastic process
$ \int_0^t e^{ A (t-s) } \, Y_s \, dW_s $,
$ t \in [0,T] $, has a
$ V_r $-valued adapted
modification.
\end{remark}
Using Lemma~\ref{lemmaRef4}
and Remark~\ref{rem:analytic}
we now present the
proof of Theorem~\ref{thm:existence}.
\begin{proof}[Proof of 
Theorem~\ref{thm:existence}]
The real number
$ R \in (0,\infty) $ 
given by
\begin{equation}
\begin{split}
  R 
  & := 
  1
+
  \big\| 
    ( \eta - A )^{ -1 } 
  \big\|_{ L( H ) }
+
  \left\| F(0) \right\|_{ H }
+
  \sup_{ \substack{v,w \in H \\ v \neq w } }
  \left(
  \frac{ 
    \left\| F( v ) - F( w ) \right\|_H 
  }{
    \left\| v - w \right\|_H 
  }
  \right) 
\\ & \qquad +
  \left\| B(0) \right\|_{ HS( U_0, H ) }
+
  \sup_{ \substack{v,w \in H \\ v \neq w } }
  \left(
  \frac{ \left\| B( v ) - B( w ) \right\|_{ HS( U_0, H ) }}{
    \left\| v - w \right\|_H } 
  \right)
\end{split}
\end{equation}
is used throughout this proof. Due to Assumptions \ref{semigroup}-\ref{diffusion} the number $ R $ is indeed finite.
Moreover, let $ \mathcal{V}_{r} $ 
for $r \in [0,\infty) $ be 
the $ \mathbb{R} $-vector space of equivalence 
classes of $ V_{r} $-valued predictable 
stochastic processes 
$ 
  Y \colon 
  [0,T] \times \Omega 
  \rightarrow V_{r} 
$ 
that satisfy
\begin{equation}\label{eqn:predictable}
  \sup_{ t \in [0,T] } 
  \mathbb{E}\big[
    \left\| Y_t 
    \right\|_{ V_r }^p  
  \big] < \infty
\end{equation} 
where two stochastic processes lie 
in one equivalence class if and only if 
they are modifications of each other. 
As usual we do not distinguish between a predictable stochastic process 
$ 
  Y \colon [0,T] \times 
  \Omega 
  \rightarrow V_{r} 
$ 
satisfying \eqref{eqn:predictable} 
and its equivalence class in 
$ \mathcal{V}_{r} $ for 
$ r \in [0,\infty) $. 
Then we equip these spaces 
with the norms 
$$
  \left\| Y \right\|_{ \mathcal{V}_{r} , u } 
  := 
  \sup_{ t \in [0,T] } 
  \left( 
    e^{ u t } 
    \left\| Y_t 
    \right\|_{ L^p ( \Omega; V_{r} ) } 
  \right) 
$$ 
for all 
$ Y \in \mathcal{V}_{r} $, 
$ u \in \mathbb{R} $ 
and all $ r \in [0,\infty) $. 
Note that 
the pair 
$ 
  \big( 
    \mathcal{V}_{r}, 
    \left\| \cdot 
    \right\|_{  \mathcal{V}_{r} , u }   
    \!
  \big) 
$ 
is an 
$ \mathbb{R} $-Banach 
space for every 
$ u \in \mathbb{R} $ and 
every $ r \in [0,\infty) $. 
In the next step we 
consider the mapping 
$ 
  \Phi \colon \mathcal{V}_{\alpha} 
  \rightarrow \mathcal{V}_{\alpha} 
$ 
given by
\begin{equation}
\label{eqn:phi}
  \left( \Phi  Y  \right)_{ t }
  :=
  e^{ At } \xi
  +
  \int_0^t
  e^{ A(t-s) }
  F\!\left( Y_s \right) ds
  +
  \int_0^t
  e^{ A(t-s) }
  B\!\left( Y_s \right) dW_s 
\end{equation}
$ \mathbb{P} $-a.s.\ for 
all $ t \in [0,T] $ 
and all 
$ Y \in \mathcal{V}_{\alpha} $.
In the following we show that 
$ 
  \Phi \colon 
  \mathcal{V}_{\alpha} 
  \rightarrow \mathcal{V}_{\alpha} 
$ 
given by \eqref{eqn:phi} 
is well defined.

To this end note 
that Assumptions \ref{semigroup} 
and \ref{initial} 
yield that 
$ 
  e^{ A t } \xi 
$,
$
  t \in [0,T]
$,
is an 
adapted $ V_{\gamma} $-valued 
stochastic process with 
continuous sample paths. 
Hence, 
$ 
  e^{ A t } \xi 
$,
$
  t \in [0,T]
$,
is a $ V_{\gamma} \subset V_{\alpha} $-valued 
predictable stochastic process 
(see Proposition 3.6 (ii) 
in \cite{dz92}). 
Additionally, we have
\begin{equation}
\label{eq:welldefined1}
\begin{split}
  \sup_{ t \in [0,T] }
  \mathbb{E}\Big[
    \left\| e^{ A t } \xi \right\|^p_{ 
      V_{\gamma} 
    }
  \Big]
& \leq 
  \left(
    \sup_{ t \in [0,T] }
    \left\| e^{ A t }
    \right\|^p_{ L( H ) }
  \right)
  \mathbb{E}\big[
    \| \xi 
    \|^p_{ V_{ \gamma } }
  \big]
\\ & \leq
  \left| c_0 \right|^p
  \mathbb{E}\big[
    \| \xi 
    \|^p_{ V_{ \gamma } }
  \big]
  < \infty ,
\end{split}
\end{equation}
which shows that 
$ 
  e^{ A t } \xi 
$,
$
  t \in [0,T]
$,
is indeed in 
$ 
  \mathcal{V}_{\gamma} 
  \subset \mathcal{V}_{\alpha} 
$.

We now concentrate
on the second summand
on the right hand side 
of \eqref{eqn:phi}.
First 
observe that the mapping 
$ F|_{ V_{ \alpha } }
\colon V_{ \alpha } \rightarrow H $
given by
$ F|_{ V_{ \alpha } }(v) = F(v) $
for all $ v \in V_{ \alpha } $
is  
$ 
  \mathcal{B}\left( V_{ \alpha } \right) 
$/$ 
  \mathcal{B}\left( H \right) 
$-measurable. Indeed, 
the Kuratowski theorem gives
$ V_{ \alpha } \in \mathcal{B}(H) $
and
$ \mathcal{B}( V_{ \alpha } )
= \mathcal{B}(H) \cap V_{ \alpha } $
which in turn implies the
asserted Borel measurability
of $ F|_{ V_{ \alpha } } $.
Next
Lemma~\ref{lemmaRef4}
and
Jensen's inequality
yield
\begin{align*}
&
  \int_0^t 
  \mathbb{E}\Big[
    \big\|
      e^{ A( t -s ) } F( Y_s ) 
    \big\|_{ V_{\gamma} } 
  \Big] \, ds
\leq 
  \int_0^t
  \big\| ( \eta - A )^{ \gamma }
    e^{ A( t -s ) }
  \big\|_{ L( H ) }
  \,
  \mathbb{E}\big[
    \| F( Y_s ) \|_H 
  \big] 
  \, ds
\\ & \leq 
  R \,
  c_{ \gamma }
  \int_0^t
  \left( t - s \right)^{
    - \gamma
  }
  \left( 
    1 + 
    \mathbb{E}\big[ \| Y_s \|_H \big]
  \right) 
  ds
\leq
  \frac{
    R \, c_{ \gamma } \,
    T^{ \left( 1 - \gamma \right) }
  }{
    \left( 1 - \gamma \right)
  }
  \left( 
    1 + 
    \sup_{ s \in [0,T] }
    \mathbb{E}\big[ \| Y_s \|_H \big]
  \right) 
\\ & \leq
  \frac{
    R \, c_{ \gamma } \,
    T^{ \left( 1 - \gamma \right) }
  }{
    \left( 1 - \gamma \right)
  }
  \left( 
    1 + 
    \sup_{ s \in [0,T] }
    \| Y_s \|_{
      L^p( \Omega; H_{ \alpha } )
    }
  \right) 
< \infty
\end{align*}
for all $ t \in [0,T] $ 
and all $ Y \in \mathcal{V}_{\alpha} $. 
This
shows that 
$ \int_0^t e^{ A( t - s ) } F( Y_s ) \, ds $, 
$ t \in [0,T] $, is a well defined $ V_{\gamma} $-valued (and in particular $ V_{ \alpha } $-valued) adapted stochastic process 
for every $ Y \in \mathcal{V}_{\alpha} $. 
Moreover, we have
\begin{align*}
 &\left\|
    \int_0^{ t_2 }
    e^{ A( t_2 - s ) }
    F( Y_s ) \, ds
    -
    \int_0^{ t_1 }
    e^{ A( t_1 - s ) }
    F( Y_s ) \, ds
  \right\|_{ L^p( \Omega; V_{r} ) }
\\&\leq 
  \left\|
    \int_{ t_1 }^{ t_2 }
    e^{ A( t_2 - s ) }
    F( Y_s ) \, ds
  \right\|_{ L^p( \Omega; V_{r} ) }
\\&
  +
  \left\|
    \left(
      e^{ A( t_2 - t_1 ) } - I
    \right)
    \int_{ 0 }^{ t_1 }
    e^{ A( t_1 - s ) }
    F( Y_s ) \, ds
  \right\|_{ L^p( \Omega; V_{r} ) }
\\&\leq 
  \int_{ t_1 }^{ t_2 }
  \big\|
    ( \eta - A )^{ r }
    e^{ A( t_2 - s ) }
  \big\|_{ L( H ) }
  \left\|
    F( Y_s ) 
  \right\|_{ L^p( \Omega; H ) } ds
\\&+
  \big\|
    ( \eta - A )^{ 
      (r - \gamma - \varepsilon) 
    }
    \,
    (
      e^{ A( t_2 - t_1 ) } - I
    )
  \big\|_{ L( H ) }
  \int_{ 0 }^{ t_1 }
  \|
    e^{ A( t_1 - s ) }
    F( Y_s ) 
  \|_{ 
    L^p( 
      \Omega; 
      V_{\gamma + \varepsilon} 
    ) 
  }
  \, ds
\end{align*}
and 
Lemma~\ref{lemmaRef4} 
thus shows
\begin{align*}
& 
  \left\|
    \int_0^{ t_2 }
    e^{ A( t_2 - s ) }
    F( Y_s ) \, ds
    -
    \int_0^{ t_1 }
    e^{ A( t_1 - s ) }
    F( Y_s ) \, ds
  \right\|_{ L^p( \Omega; V_{r} ) }
\\&\leq
  c_r 
  \int_{ t_1 }^{ t_2 }
  \left( t_2 - s \right)^{ - r }
  \left\|
    F( Y_s ) 
  \right\|_{ L^p( \Omega; H ) } ds
\\ & \quad +
  c_{ 
    \left( \gamma + \varepsilon - r \right) 
  }
  \,
  c_{
    \left( \gamma + \varepsilon \right)
  }
  \left( 
    t_2 - t_1 
  \right)^{ 
    \left( \gamma + \varepsilon - r \right) 
  }
  \int_{ 0 }^{ t_1 }
  \left( t_1 - s \right)^{
    - 
    \left( 
      \gamma + \varepsilon 
    \right)
  }
  \|
    F( Y_s ) 
  \|_{ 
    L^p( \Omega; H ) 
  } 
  \,
  ds
\\ & \leq
  R
  \left(
  \frac{ 
    c_r 
    \left( 
      t_2 - t_1
    \right)^{ 
      \left( 1 - r \right) 
    }
  }{
    \left( 1 - r \right)
  }
  +
  \frac{  
    c_{ 
      \left( \gamma + \varepsilon - r \right) 
    }
    \,
    c_{
      \left( \gamma + \varepsilon \right)
    }
    \left( t_2 - t_1 \right)^{
      \left( \gamma + \varepsilon - r \right)
    }
    T^{ 
      \left( 1 - \gamma - \varepsilon \right) 
    }
  }{ 
    \left( 1 - \gamma - \varepsilon \right)
  }
  \right)
\\ & \quad \cdot
  \left(
    1 + \sup_{ s \in [0,T] }
    \left\| Y_s \right\|_{ L^p( \Omega; H ) }
  \right)
\end{align*}
for all 
$ t_1, t_2 \in [0,T] $ 
with $ t_1 \leq t_2 $, 
$ \varepsilon \in [0,1-\gamma) $,
$ r \in [0,\gamma] $ 
and all 
$ Y \in \mathcal{V}_{\alpha} $. 
This finally
shows
\begin{multline}
\label{eq:welldefined3}
 \left\|
    \int_0^{ t_2 }
    e^{ A( t_2 - s ) }
    F( Y_s ) \, ds
    -
    \int_0^{ t_1 }
    e^{ A( t_1 - s ) }
    F( Y_s ) \, ds
  \right\|_{ L^p( \Omega; V_{r} ) }
\\
\leq
  R
  \left(
  \frac{  
    T^{ 
      \left( 1 - \gamma - \varepsilon \right) 
    }
  }{ 
    \left( 1 - \gamma - \varepsilon \right)
  }
  \right)
  \left(
    1 + \sup_{ s \in [0,T] }
    \left\| Y_s \right\|_{ L^p( \Omega; H ) }
  \right)
\\
  \cdot
    \left(
      c_r 
      +
      c_{ 
        \left( \gamma + \varepsilon - r \right) 
      }
      \,
      c_{
        \left( \gamma + \varepsilon \right)
      }
    \right)
    \left( t_2 - t_1 \right)^{
      \left( \gamma + \varepsilon - r \right)
    }
\end{multline}
for all $ t_1, t_2 \in [0,T] $ 
with $ t_1 \leq t_2 $, 
$ \varepsilon \in [0,1-\gamma) $,
$ r \in [0,\gamma ] $
and all $ Y \in \mathcal{V}_{\alpha} $. 
Proposition 3.6 (ii) in \cite{dz92} thus yields 
that the stochastic 
process $ \int_0^t e^{ A( t-s ) } F( Y_s ) \, ds $, 
$ t \in [0,T] $, has a modification 
in $ \mathcal{V}_{\gamma} \subset \mathcal{V}_{\alpha} $ 
for every $ Y \in \mathcal{V}_{\alpha} $. 

In the sequel
we concentrate
on the third summand
on the right hand side 
of \eqref{eqn:phi}. 
First observe that 
Kuratowski's theorem 
shows
$ V_{ \alpha } \in \mathcal{B}(H) $,
$
  HS( U_0, V_{ \alpha } )
  \in
  \mathcal{B}( HS( U_0, H ) )
$,
$ \mathcal{B}( V_{ \alpha } )
= \mathcal{B}(H) \cap V_{ \alpha } $
and
$
  \mathcal{B}( 
    HS( U_0, V_{ \alpha } ) 
  )
  =
  \mathcal{B}( 
    HS( U_0, H ) 
  )
  \cap
  HS( U_0, V_{ \alpha } ) 
$.
This implies that
the mapping 
$ \tilde{B} \colon V_{ \alpha }
\rightarrow HS( U_0, V_{ \alpha } ) $
given by
$ \tilde{B}(v) = B(v) $
for all $ v \in V_{ \alpha } $
is
$
  \mathcal{B}( V_{ \alpha } )
$/$
  \mathcal{B}( HS( U_0, V_{ \alpha } ) )
$-measurable.
Next 
Lemma \ref{lemmaRef4} gives
\begin{align*}
  &
  \int_0^t
  \mathbb{E}\Big[
    \big\|
      e^{ A( t - s ) } B( Y_s )
    \big\|^2_{ HS( U_0, V_{ \gamma } ) } 
  \Big] \, ds
\\ & \leq
  \int_0^t
  \big\|
    ( \eta - A 
    )^{ ( \gamma - \alpha ) } 
    \,
    e^{ A( t - s ) }
  \big\|^2_{ L ( H ) } \,
  \mathbb{E}\big[
    \|
    B( Y_s )
    \|^2_{ HS( U_0, V_{\alpha} ) } 
  \big] \, ds
\\ & \leq 
  2 \, c^2 \,
  |
    c_{
      \left( \gamma - \alpha \right)
    }
  |^2
  \int_0^t
  \left( t - s \right)^{ 
    ( 2 \alpha - 2 \gamma ) 
  }
  \left( 
    1 +
    \mathbb{E}\big[
      \|
        Y_s
      \|^2_{ V_{\alpha} }
    \big]
  \right) ds 
\\ & \leq 
  \left(
  \frac{ 
    2 \, c^2 \,
    |
      c_{
        \left( \gamma - \alpha \right)
      }
    |^2
    \,
    T^{ ( 1 + 2 \alpha - 2 \gamma ) } }{
    \left( 1 + 2 \alpha - 2 \gamma \right) }
  \right)
  \left( 
    1 + \sup_{ s \in [0,T] }
    \mathbb{E}\big[ \|
      Y_s
    \|^2_{ V_{\alpha} } \big]
  \right)
  < \infty
\end{align*}
for all $ t \in [0,T] $ 
and all $ Y \in \mathcal{V}_{\alpha} $.
Therefore, 
Remark~\ref{rem:analytic}
shows that 
$ 
  \int_0^t 
$
$
    e^{ A( t - s ) } B( Y_s ) 
$
$ 
  dW_s 
$, 
$ t \in [0,T] $, is a well defined $ V_{\gamma} $-valued (and in particular $ V_{ \alpha } $-valued) adapted stochastic process for every $ Y \in \mathcal{V}_{\alpha} $
(cf. the heuristic 
calculation~\eqref{eq:newX1} in the
introduction).
Moreover, the Burkholder-Davis-Gundy type inequality in Lemma~7.7 in
\cite{dz92} gives
\begin{align*}
&
  \left\|
    \int_0^{ t_2 }
    e^{ A( t_2 - s ) } B( Y_s ) \, dW_s
    -
    \int_0^{ t_1 }
    e^{ A( t_1 - s ) } B( Y_s ) \, dW_s
  \right\|_{ 
    L^p( \Omega; V_{ r } ) 
  }
\\ & \leq 
  \left\|
    \int_{ t_1 }^{ t_2 }
    e^{ A( t_2 - s ) } B( Y_s ) \, dW_s
  \right\|_{ L^p( \Omega; V_{ r } ) }
\\ & +
  \left\|
    \left( e^{ A( t_2 - t_1 ) } - I \right)
    \int_0^{ t_1 }
    e^{ A( t_1 - s ) } B( Y_s ) \, dW_s
  \right\|_{ L^p( \Omega; V_{ r } ) }
\\ & \leq 
  p 
  \left(
    \int_{ t_1 }^{ t_2 }
    \big\|
      e^{ A( t_2 - s ) } B( Y_s )
    \big\|^2_{ 
      L^p( \Omega; HS( U_0, V_{r} ) ) 
    } \, ds
  \right)^{ \! \frac{1}{2} }
\\ & +
  p \,
  \big\|
    e^{ A( t_2 - t_1 ) } - I
  \big\|_{ 
    L( H, 
    V_{ 
      ( r - \gamma - \varepsilon) 
    } ) 
  }
  \left(
    \int_0^{ t_1 } \!
    \big\|
      e^{ A( t_1 - s ) } B( Y_s )
    \big\|^2_{ 
      L^p( \Omega; 
        HS( U_0, V_{\gamma + \varepsilon} ) 
      ) 
    } \,
    ds
  \right)^{\!\frac{1}{2} }
\end{align*}
and 
Lemma~\ref{lemmaRef4} 
therefore shows
\begin{align}
\label{eq:implied_by_L1}
&
  \left\|
    \int_0^{ t_2 }
    e^{ A( t_2 - s ) } B( Y_s ) \, dW_s
    -
    \int_0^{ t_1 }
    e^{ A( t_1 - s ) } B( Y_s ) \, dW_s
  \right\|_{ L^p( \Omega; V_{ r } ) }
\nonumber
\\ & \leq 
\nonumber
  p 
  \left(
    \int_{ t_1 }^{ t_2 }
    \big\|
      ( \eta - A )^{ 
        ( r - \alpha ) 
      }
      \,
      e^{ A (t_2 - s) }
    \big\|_{ L(H) }^2
    \left\| 
      B( Y_s )
    \right\|^2_{ 
      L^p( \Omega; HS( U_0, V_{ \alpha } ) 
    ) }  
    ds
  \right)^{ \! \frac{1}{2} }
\\ & \quad +
  p \,  
  c_{ ( \gamma + \varepsilon - r ) }
  \,
  c_{ ( \gamma + \varepsilon - \alpha ) } 
  \left( t_2 - t_1 
  \right)^{ 
    ( \gamma + \varepsilon - r ) 
  }
\\ & \quad \quad \cdot
\nonumber
  \left(
    \int_0^{ t_1 } 
    \left( t_1 - s \right)^{
      ( 2 \alpha - 2 \gamma - 2 \varepsilon )
    }
    \left\|
      B( Y_s )
    \right\|^2_{ 
      L^p( \Omega; HS( U_0, V_{ \alpha } ) 
      ) 
    } ds
  \right)^{ \! \frac{1}{2} }
\end{align} 
for all $ t_1, t_2 \in [0,T] $ 
with 
$ t_1 \leq t_2 $,
$ \varepsilon \in 
[0, \frac{1}{2} + \alpha - \gamma ) $,
$ r \in [0, \gamma ] $ 
and all
$ Y \in \mathcal{V}_{ \alpha } $.
In the case $ r \in [\alpha,\gamma] $ 
we have
\begin{equation}\label{eqn:caseA}
  \big\|
    ( \eta - A )^{ ( r - \alpha ) } \,
    e^{ A s }
  \big\|_{ L( H ) }
  \leq 
  c_{ ( r - \alpha ) }
  \,
  s^{ (\alpha - r ) }
\end{equation}
for all $ s \in (0,T] $ 
(see Lemma \ref{lemmaRef4}) 
and in the case 
$ r \in [0, \alpha) $ 
we have
\begin{equation}\label{eqn:caseB}
  \big\|
    ( \eta - A )^{ ( r - \alpha ) }
    \,
    e^{ A s }
  \big\|_{ L( H ) }
  \leq 
  \big\|
    ( \eta - A )^{ ( r - \alpha ) }
  \big\|_{ L( H ) }
  \,
  c_0
  \leq 
  c_0 \, R
\end{equation}
for all $ s \in (0,T] $.
Combining \eqref{eqn:caseA} 
and \eqref{eqn:caseB} 
shows
\begin{equation}
\begin{split}
\lefteqn{
  \left(
  \int_0^t
  \big\|
    ( \eta - A )^{ ( r - \alpha ) }
    \,
    e^{ A s }
  \big\|_{ L( H ) }^2
  \, 
  ds
  \right)^{ \! \frac{ 1 }{ 2 } }
}
\\ & \leq
  \left(
  \int_0^t
  \left(
    | 
      c_{ \max( r - \alpha, 0 ) } 
    |^2 \,
    s^{ 
      ( 2 \alpha - 2 r )
    }
    +
    | 
      c_{ \max( r - \alpha, 0 ) } 
    |^2 \,
    R^2
  \right)
  ds
  \right)^{ \! \frac{ 1 }{ 2 } }
\\ & \leq
  c_{ \max( r - \alpha, 0 ) } 
  \,
  R
  \left(
    \frac{
      t^{ 
        ( 1/2 + \alpha - r )
      }
    }{
      \sqrt{ 
        1 + 2 \alpha - 2 r
      }
    }
    +
    t^{ 1/2 }
  \right)
\end{split}
\end{equation}
and hence
\begin{equation}
\label{eq:xxx}
\begin{split}
\lefteqn{
  \left(
  \int_0^t
  \big\|
    ( \eta - A )^{ ( r - \alpha ) }
    \,
    e^{ A s }
  \big\|_{ L( H ) }^2
  \, 
  ds
  \right)^{ \! \frac{ 1 }{ 2 } }
}
\\ & \leq
  \frac{
  c_{ \max( r - \alpha, 0 ) } 
  \,
  R
  \left(
    t^{ 
      ( 1/2 + \alpha - r )
    }
    +
    t^{ 1/2 }
  \right)
  }{
    \sqrt{ 
      1 + 2 \alpha - 2 \gamma 
      - 2 \varepsilon
    }
  }
\\ & \leq
  \frac{
  c_{ \max( r - \alpha, 0 ) } 
  \,
  R
  \left(
    T^{ 
      ( 
        1/2 + \alpha 
        - \gamma - \varepsilon
      ) 
    }
    \,
    t^{ 
      ( 
        \gamma + \varepsilon - r
      )
    }
    +
    t^{ 1/2 }
  \right)
  }{
    \sqrt{ 
      1 + 2 \alpha - 2 \gamma 
      - 2 \varepsilon
    }
  }
\end{split}
\end{equation}
for all 
$ t \in [0,T] $,
$ 
  \varepsilon \in 
  [0, \frac{1}{2} + \alpha - \gamma ) 
$
and all
$ 
  r \in [0, \gamma ] 
$.
Using \eqref{eq:xxx} in
\eqref{eq:implied_by_L1} 
then gives
\begin{align}
\label{eq:deduced}
&
  \left\|
    \int_0^{ t_2 }
    e^{ A( t_2 - s ) } B( Y_s ) \, dW_s
    -
    \int_0^{ t_1 }
    e^{ A( t_1 - s ) } B( Y_s ) \, dW_s
  \right\|_{ L^p( \Omega; V_{ r } ) }
\nonumber
\\ & \leq 
\nonumber
  \Bigg(
  \left(
  \frac{
    2 
    \, p
    \,
    c_{ \max( r - \alpha, 0 ) } 
    \,
    R 
    \,
    \max(T,1)
    \left( t_2 - t_1 \right)^{ 
      \min( 
        \gamma + \varepsilon - r, \frac{ 1 }{ 2 }
      )
    }
  }{
    \sqrt{ 
      1 + 2 \alpha - 2 \gamma 
      - 2 \varepsilon
    }
  }
  \right) 
\\ & \quad +
  \left(
    \frac{
      p \,  
      c_{ ( \gamma + \varepsilon - r ) }
      \,
      c_{ ( \gamma + \varepsilon - \alpha ) } 
      \max(T,1)
      \left( t_2 - t_1 
      \right)^{ 
        \min( \gamma + \varepsilon - r,
        \frac{ 1 }{ 2 } ) 
      }
    }{
      \sqrt{ 
        1 + 2 \alpha - 2 \gamma 
        - 2 \varepsilon
      }
    }
  \right)
  \Bigg)
\\ & \quad \quad \cdot
\nonumber
  \left(
    \sup_{ t \in [0,T] }
    \left\| 
      B( Y_t )
    \right\|_{ 
      L^p( \Omega; HS( U_0, V_{ \alpha } ) 
    ) }  
  \right)
\end{align} 
for all $ t_1, t_2 \in [0,T] $ 
with 
$ t_1 \leq t_2 $,
$ 
  \varepsilon \in 
  [0, \frac{1}{2} + \alpha - \gamma ) 
$,
$ 
  r \in [0, \gamma ] 
$ 
and all
$ Y \in \mathcal{V}_{ \alpha } $.
Proposition 3.6 (ii) 
in \cite{dz92} 
thus yields that 
$ 
  \int_0^t 
  e^{ A( t - s ) } B( Y_s ) 
  \, dW_s 
$, 
$ t \in [0,T] $, has a modification 
in 
$
  \mathcal{V}_{\gamma} \subset 
  \mathcal{V}_{\alpha} 
$
for every 
$ 
  Y \in 
  \mathcal{V}_{ \alpha } 
$
%%%%%%%%%%%%%%%%
and this finally shows the well 
definedness of 
$ 
  \Phi \colon 
  \mathcal{V}_{\alpha} \rightarrow 
  \mathcal{V}_{\alpha} 
$ 
in \eqref{eqn:phi}
(see
\eqref{eq:welldefined1}, \eqref{eq:welldefined3} and \eqref{eq:deduced}).
%%%%%%%%%%%%%%
%%%%%%%%%%%%%%
%%%%%%%%%%%%%%%
%%%%%%%%%%%%%%

In the next step we show that 
$ 
  \Phi \colon \mathcal{V}_{\alpha} 
  \rightarrow \mathcal{V}_{\alpha} 
$ 
is a contraction with respect 
to 
$ 
  \left\| \cdot 
  \right\|_{ \mathcal{V}_{\alpha}, u } 
$ 
for an appropriate $ u \in \mathbb{R} $. 
The Banach fixed point theorem will 
then yield the existence of a unique 
fixed point for 
$ 
  \Phi \colon \mathcal{V}_{\alpha} 
  \rightarrow \mathcal{V}_{\alpha} 
$. 
More formally, 
Lemma 7.7 in \cite{dz92} gives
\begin{align*}
 &\left\|
    \left( \Phi Y \right)_t 
    - 
    \left( \Phi Z \right)_t
  \right\|_{ L^p( \Omega; V_{\alpha} ) }
\\&\leq 
  \left\|
    \int_0^t
    e^{ A( t - s ) }
    \left( F( Y_s ) - F( Z_s ) \right) ds
  \right\|_{ L^p( \Omega; V_{\alpha} ) }
\\&\quad+
  \left\|
    \int_0^t
    e^{ A( t - s ) }
    \left( B( Y_s ) - B( Z_s ) \right) dW_s
  \right\|_{ L^p( \Omega; V_{\alpha} ) }
\\&\leq
  \int_0^t 
  \big\|
    ( \eta - A )^{ \alpha } \,
    e^{ A( t - s ) }
  \big\|_{ L( H ) }
  \left\|
    F( Y_s ) - F( Z_s )
  \right\|_{ L^p( \Omega; H ) } ds
\\&\quad+
  p \left(
    \int_0^t
    \big\|
      ( \eta - A )^{ \alpha } \,
      e^{ A( t - s ) }
    \big\|_{ L( H ) }^2
    \left\|
      B( Y_s ) - B( Z_s ) 
    \right\|^2_{ 
      L^p( \Omega; HS( U_0, H ) ) 
    } 
    ds
  \right)^{ 
    \! \frac{1}{2} 
  }
\end{align*}
and the definition of $R$ 
and Lemma~\ref{lemmaRef4}
yield
\begin{align*}
 &\left\|
    \left( \Phi Y \right)_t 
    - 
    \left( \Phi Z \right)_t
  \right\|_{ L^p( \Omega; V_{\alpha} ) }
\\&\leq
  R \, 
  c_{
    \alpha
  }
  \int_0^t 
  \left( t - s \right)^{
    - \alpha
  }
  \left\|
    Y_s - Z_s
  \right\|_{ L^p( \Omega; H ) } ds
\\ & \quad +
  p \, R \, c_{ \alpha }
  \left(
    \int_0^t
    \left( t - s \right)^{ - 2 \alpha }
    \left\|
      Y_s - Z_s
    \right\|^2_{ L^p( \Omega; H ) } 
    ds
  \right)^{ \frac{1}{2} }
\\ & \leq
  R \,
  c_{ \alpha }
  \left(
    \int_0^t
    \left( t - s \right)^{ -\alpha } 
    e^{ -u s } \, ds
  \right)
  \left\|
    Y - Z
  \right\|_{ \mathcal{V}_0, u }
\\ & \quad +
  p \, R \, c_{ \alpha }
  \left(
    \int_0^t
    \left( t - s \right)^{ -2\alpha }
    e^{ - 2 u s } \, ds
  \right)^{ 
    \! \frac{1}{2} 
  }
  \left\|
    Y - Z
  \right\|_{ \mathcal{V}_0, u }
\end{align*}
for all $ t \in [0,T] $,
$ 
  Y, Z \in \mathcal{V}_{\alpha} 
$ 
and 
all $ u \in \mathbb{R} $. 
The Cauchy-Schwartz inequality
therefore implies
\begin{equation}
\begin{split}
\lefteqn{
  \left\|
    \Phi( Y ) 
    - 
    \Phi( Z ) 
  \right\|_{ \mathcal{V}_{\alpha}, u }
}
\\ & \leq
  R \, 
  c_{ 
    \alpha
  } 
  \left( \sqrt{T} + p \right)
  \left(
    \int_0^T
    s^{ -2 \alpha } \,
    e^{ 2 u s } \, ds
  \right)^{ \! \frac{1}{2} }
  \left\|
    Y - Z
  \right\|_{ \mathcal{V}_0, u }
\\ & \leq
  R \, c_{ \alpha } 
  \left( \sqrt{T} + p \right)
  \left(
    \int_0^T
    \frac{ e^{ 2 u s } }{
      s^{ 2 \alpha }
    }
    \, ds
  \right)^{ \! \frac{1}{2} }
  \left\|
    Y - Z
  \right\|_{ \mathcal{V}_{ \alpha }, u }
\end{split}
\end{equation}
for all 
$
  Y, Z \in \mathcal{V}_{\alpha} 
$ 
and 
all
$ 
  u \in \mathbb{R} 
$. 
This shows that 
$ 
  \Phi \colon \mathcal{V}_{\alpha} 
  \rightarrow \mathcal{V}_{\alpha} 
$ 
is a contraction with 
respect to 
$ \left\| \cdot \right\|_{ \mathcal{V}_{ \alpha }, u } $ 
for a sufficiently small $ u \in (-\infty, 0 ) $. Hence, there is a 
up to modifications
unique predictable stochastic
process 
$ 
  Y \colon [0,T] \times \Omega
  \rightarrow V_{ \alpha }
  \in \mathcal{V}_{\alpha} 
$
with $ \Phi( Y ) = Y $, i.e. 
\begin{equation}\label{eq:SPDEsol}
    Y_t 
    =
    e^{ At } \xi
    + 
    \int_0^t
    e^{ A ( t - s ) } F( Y_s ) \, ds
    +
    \int_0^t
    e^{ A ( t - s ) } B( Y_s ) \, dW_s
\end{equation}
$ \mathbb{P} $-a.s.\ for 
all $ t \in [0,T] $. 
Moreover, 
% the fact
% $ Y \in \mathcal{V}_{ \alpha } $, 
\eqref{eq:welldefined1}, \eqref{eq:welldefined3},
\eqref{eq:deduced} 
and Proposition~3.6
(ii) in \cite{dz92}
then
show that
there exists a
predictable modification
$ 
  X \colon [0,T]
  \times \Omega
  \rightarrow V_{ \gamma }
$
of
$ 
  Y \colon [0,T]
  \times \Omega
  \rightarrow V_{ \alpha }
$.

Additionally, 
note that 
the inequality
$
  \left\|
    B(v)
  \right\|_{
    HS( U_0, H_{ \alpha } )
  }^p
  \leq
  2^p \,
  c^p 
  \left(
    1 +
    \left\| v \right\|_{ H_{ \alpha } }^p 
  \right)
$
for all
$ v \in H_{ \alpha } $
(see Assumption~\ref{diffusion})
implies
\begin{align}\label{eq:Balpha}
  \sup_{ t \in [0,T] }
  \mathbb{E}\Big[
    \big\|
    B( X_t ) 
    \big\|^p_{ HS( U_0, V_{\alpha} ) }
  \Big]
\leq 
  2^p \, c^p
  \left(
    1 + \sup_{ t \in [0,T] }
    \mathbb{E}\big[ 
      \| X_t \|^p_{ V_{\alpha } }
    \big]
  \right)
  < \infty .
\end{align}
It remains to establish
the temporal continuity 
properties asserted
in Theorem~\ref{thm:existence}.
To this end note that
Lemma~\ref{lemmaRef4}
gives
\begin{align}\label{eq:last}
&
  \left\|
    e^{ A t_2 } \xi -
    e^{ A t_1 } \xi 
  \right\|_{ L^p( \Omega; V_r ) }
\nonumber
\\&
=
  \left\|
    e^{ A t_1 }
    \left( \eta - A 
    \right)^{ \left( r - \gamma 
      \right) 
    }
    \left(
      e^{ A ( t_2 - t_1 ) }
      - I
    \right)
    \left( \eta - A 
    \right)^{ \gamma }
    \xi
  \right\|_{ L^p( \Omega; H ) }
\\&\leq
\nonumber
  \left\| e^{ A t_1 }
  \right\|_{ L(H) }
  \big\|
    ( \eta - A 
    )^{ 
      ( r - \gamma
      ) 
    }
    \,
    (
      e^{ A ( t_2 - t_1 ) }
      - I
    ) 
  \big\|_{ L(H) }
  \left\| \xi \right\|_{ 
    L^p( \Omega; V_{ \gamma } ) 
  }
\\ & \leq
\nonumber
  c_0 \,
  c_{ 
    \left( 
      \gamma - r 
    \right)
  }
  \left\| \xi \right\|_{ 
    L^p( \Omega; V_{ \gamma } ) 
  }
  \left( t_2 - t_1 \right)^{
    \left( \gamma - r \right)
  }
\end{align}
for all $ t_1, t_2 \in [0,T] $
with $ t_1 \leq t_2 $
and all 
$ r \in [0,\gamma] $.
Combining \eqref{eq:welldefined3}, 
\eqref{eq:deduced}
and \eqref{eq:last} 
then yields
\eqref{eq:holder}.
Finally, \eqref{eq:welldefined1},
\eqref{eq:welldefined3}
and \eqref{eq:deduced}
show that $ X_t $, $ t \in [0,T] $,
is continuous with respect to
$ 
  \big( \mathbb{E}\big[ 
      \| \! \cdot \!
      \|_{ V_{ \gamma } }^p
    \big]
  \big)^{ \frac{1}{p} } 
$.
This completes the proof 
of Theorem~\ref{thm:existence}. 
\end{proof}
\subsubsection*{Acknowledgement}
This work has been 
partially supported by the
research project
``Numerical solutions of
stochastic differential equations
with non-globally
Lipschitz continuous 
coefficients'',
by the Collaborative 
Research Centre $701$
``Spectral Structures 
and Topological Methods 
in Mathematics'',
by the
International Graduate 
School ``Stochastics and Real 
World Models''
(all funded by the German
Research Foundation) and
by the 
BiBoS Research Center.
The support of Issac Newton
Institute for Mathematical
Sciences in Cambridge is
also gratefully acknowledged
where part of this was done
during the special semester
on ``Stochastic Partial Differential
Equations''. Additionally,
we would like to express
our sincere gratitude to Sonja
Cox, Jan van Neerven 
and an anonymous referee
for their
very helpful advice.
\bibliographystyle{acm}
\bibliography{bibfile}
\end{document}